\documentclass[1p]{article}

\usepackage{amssymb}
\usepackage{amsthm}

\usepackage{amsmath}
\usepackage{mathtools}
\usepackage{float}
\usepackage{tikz}
\usetikzlibrary{arrows.meta}
\usetikzlibrary{decorations.markings}
\usepackage{pict2e} 

\usepackage{subcaption}
\usepackage[percent]{overpic}
\usetikzlibrary{shapes.geometric}

\usepackage{lipsum}
\usepackage{amsfonts}
\usepackage{epstopdf}
\usepackage{algorithmic}
\usepackage{hyperref}
\ifpdf
  \DeclareGraphicsExtensions{.eps,.pdf,.png,.jpg}
\else
  \DeclareGraphicsExtensions{.eps}
\fi

\usepackage{enumitem}
\setlist[enumerate]{leftmargin=.5in}
\setlist[itemize]{leftmargin=.5in}

\newtheorem{definition}{Definition}
\newtheorem{theorem}{Theorem}
\newtheorem{remark}{Remark}
\newtheorem{lemma}{Lemma}
\newtheorem{corollary}{Corollary}

\usepackage{amsopn}

\DeclareMathOperator{\interior}{int}
\DeclareMathOperator{\dom}{dom}

\begin{document}
\def\stablemanifold#1#2#3#4#5{\begin{scope}[thick,decoration={
    markings,
    mark=at position #4 with {\arrow{Latex[length=3mm]}},
    mark=at position #5 with {\arrowreversed{Latex[length=3mm]}}}
    ] 
    \draw[#3, postaction={decorate}] #1 -- #2;
\end{scope}}

\def\unstablemanifold#1#2#3#4#5{\begin{scope}[thick,decoration={
    markings,
    mark=at position #4 with {\arrowreversed{Latex[length=3mm]}},
    mark=at position #5 with {\arrow{Latex[length=3mm]}}}
    ] 
    \draw[#3, postaction={decorate}] #1 -- #2;
\end{scope}}

\def\stablemanifoldpositive#1#2#3#4#5{\begin{scope}[thick,decoration={
    markings,
    mark=at position #5 with {\arrowreversed{Latex[length=3mm]}}}
    ] 
    \draw[#3, postaction={decorate}] #1 -- #2;
\end{scope}}

\def\unstablemanifoldpositive#1#2#3#4#5{\begin{scope}[thick,decoration={
    markings,
    mark=at position #5 with {\arrow{Latex[length=3mm]}}}
    ] 
    \draw[#3, postaction={decorate}] #1 -- #2;
\end{scope}}

\def\hset#1#2#3#4#5{
    \draw[double=#5, thick, double distance=2pt] #1 -- #2;
    \draw[double=#5, thick, double distance=2pt] #3 -- #4;
    \draw[thick] #2 -- #3;
    \draw[thick] #1 -- #4;
}

\def\hsetrect#1#2#3#4#5{
    \hset{(#1,#3)}{(#1,#4)}{(#2,#4)}{(#2,#3)}{#5}
}

\def\hsetparallelogram#1#2#3#4#5{
    \hset
        {(#1*5/6-#3*1/6,#3*5/6-#1*1/6)}
        {(#1*5/6-#4*1/6,#4*5/6-#1*1/6)}
        {(#2*5/6-#4*1/6,#4*5/6-#2*1/6)}
        {(#2*5/6-#3*1/6,#3*5/6-#2*1/6)}{#5}
}

\begin{center}
{\Large \bf Oscillatory motion to collision and infinity in the Earth-Moon restricted three body problem}

 \vskip 0.5cm
{\large Aleksander Pasiut \footnote{The author has been partially supported by the NCN grant 2021/41/B/ST1/00407.}}

 \vskip 0.2cm
 { Faculty of Applied Mathematics,\\
 AGH University of Kraków}\\
{\small  Mickiewicza 30, 30-059 Krak\'ow, Poland} \\
e-mail: \texttt{pasiut@agh.edu.pl}

\vskip 0.5cm

 \today
\vskip 0.5cm

\end{center}

\begin{abstract}
    We present a computer assisted proof of the existence of the oscillatory motion to collision and to infinity in the Earth-Moon planar circular restricted three body problem. In particular, we show that all types of motion (hyperbolic, parabolic, oscillatory to infinity, oscillatory to collision and terminating in collision) can be combined. We also prove the existence of arbitrarily large ejection/collision orbits. The results are achieved with the use of the Levi-Civita regularization, the McGehee regularization, topological tools and rigorous numerical computations.
\end{abstract}

\vspace{0.3cm} \noindent {\bf Keywords:}  celestial mechanics, collisions, computer assisted proofs

\vspace{0.3cm} \noindent {\bf AMS classification:} 37C29, 37J46, 70F07

\vskip\baselineskip

\section{Introduction}
Planar circular restricted three body problem is a model of celestial mechanics, which describes the motion of a masless particle under the gravitational pull of two large masses, called primaries, which move on circular orbits on the same plane as the massless particle. In 1922, J. Chazy \cite{MR1509241} classified the possible final motions that a trajectory $q(t)$ of a massless body might have, when the time $t$ goes to positive or negative infinity as:
\begin{itemize}
    \item $\mathcal{O}_H^{\pm}$ (hyperbolic):
    \[
        r(t)\to\infty \quad\text{and}\quad \| \dot q(t)\| \to c>0 \quad\text{for}\quad t\to \pm \infty,
    \]
    \item $\mathcal{O}_P^{\pm}$ (parabolic):
    \[
        r(t)\to\infty \quad\text{and}\quad \| \dot q(t)\| \to 0 \quad\text{for}\quad t\to \pm \infty,
    \]
    \item $\mathcal{O}_B^{\pm}$ (bounded):
    \[
        \limsup_{t\to \pm \infty} r(t) < +\infty,
    \]
    \item $\mathcal{O}_{OS}^{\pm}$ (oscillating): 
    \[
        \limsup_{t \to \pm \infty} r(t)= + \infty \quad\text{and}\quad
        \liminf_{t \to \pm \infty} r(t) < \infty.
    \]
\end{itemize}
where we define $r(t)$ as the distance from $q(t)$ to the closest primary.

We extend the classification above by distinguishing the additional cases included in sets $\mathcal{O}_B^\pm$ and $\mathcal{O}_{OS}^\pm$ and considering the possibility of the ejection/collision:
\begin{itemize}
    \item $\mathcal{O}_{OI}^{\pm}$ (oscillating to infinity): 
    \[
        \limsup_{t \to \pm \infty} r(t)= + \infty \quad\text{and}\quad
        \liminf_{t \to \pm \infty} r(t) \in (0, +\infty),
    \]
    \item $\mathcal{O}_{OC}^{\pm}$ (oscillating to collision):
	\[
        \liminf_{t \to \pm \infty} r(t) = 0  \quad\text{and}\quad \limsup_{t \to \pm \infty} r(t) \in (0,+\infty),
    \]
    \item $\mathcal{O}_{OB}^{\pm}$ (oscillating to infinity and collision):
	\[
        \limsup_{t \to \pm \infty} r(t)= + \infty \quad\text{and}\quad \liminf_{t \to \pm \infty} r(t) = 0,
    \]
    \item $\mathcal{O}_A^{\pm}$ (bounded away from collision):
    \[
        \limsup_{t\to \pm \infty} r(t) \in (0,+\infty) \quad\text{and}\quad \liminf_{t \to \pm \infty} r(t) \in (0,+\infty),
    \]
    \item $\mathcal{O}_C^{\pm}$ (ejection/collision): 
    \[
        r(t)\to 0 \quad\text{for}\quad t\to t_c^{\pm},\quad\text{for some ejection/collision time}\quad t_c.
    \]
\end{itemize}

In this work, we present the computer-assisted proof that all of these types of motion co-exist and can be combined together. We consider the energy levels and parameter domains relevant to celestial mechanics, which are in the {\em non-perturbative} regime. Our result is formally stated in the following theorem.

\begin{theorem}
    \label{thm:primary_result}
    We consider planar circular restricted three body problem with the mass ratio parameter $\mu=1/82$, which is an approximation of the Earth-Moon system. Then,
    \begin{equation}
        \label{eq:main}
        \mathcal{O}_X^- \cap \mathcal{O}_Y^+ \ne \emptyset,
    \end{equation}
    for $X,Y \in \{H,P,OI,OC,OB,B,C\}$. 

    In particular, we prove that for every real $\varepsilon, \varepsilon' > 0$, there exists a periodic orbit satisfying:
    \[
        \max_{t\in \mathbb{R}} r(t) > \varepsilon \qquad \mbox{and} \qquad \min_{t\in \mathbb{R}} r(t) < \varepsilon'
    \]
    and there exists an ejection/collision orbit satisfying:
    \[
        \max_{t \in [t_e,t_c]} r(t) > \varepsilon,
    \]
    where $t_e,t_c \in \mathbb{R}$ are the ejection time and the collision time respectively.
\end{theorem}

The proof relies on the topological tools such as the covering relations and rigorous numerical computations performed with the use of interval arithmetic. We define the suitable h-sets which are connected by the sequences of the covering relations, which imply the existence of the orbits in question.

Our construction can be briefly descibed in the following way. We distinguish four h-sets:
\begin{itemize}
    \item $\mathcal{N}_0$ - the h-set for which the collision manifold is both a horizonal and a vertical disc,
    \item $\mathcal{N}_2$ - the h-set which is located away from the collision,
    \item $\mathcal{Q}_k$ - the sequence of h-sets which does not intersect with the collision manifold, but which approaches arbitrarily close to it for the integer $k$ going to infinity,
    \item $\mathcal{W}_n$ - the sequence of h-sets which gets arbitrarily far from the primaries for the integer $n$ going to infinity.
\end{itemize}
Then, we construct the sequences of the covering relations between these sets in a way presented in the figures \ref{fig:covering_relations_overview} and \ref{fig:covering_relations_overview_detailed}.
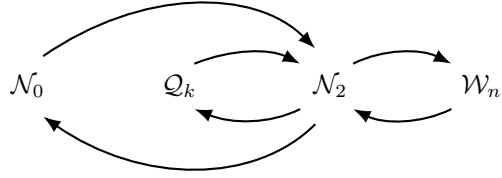
\begin{figure}[ht]
    \centering
    \begin{tikzpicture}
    \draw[-Latex, thick] (-2.8,0.4) .. controls (-1.5,1.3) and (0,1.3) .. (0.8,0.5);
    \draw[-Latex, thick] (-0.8,0.3) .. controls (-0.3, 0.5) and (0.2,0.5) .. (0.6,0.3);
    \draw[-Latex, thick] (1.3,0.3) .. controls (1.7, 0.5) and (2.2,0.5) .. (2.6,0.3);
    \draw[Latex-, thick] (-2.8,-0.4) .. controls (-1.5,-1.3) and (0,-1.3) .. (0.8,-0.5);
    \draw[Latex-, thick] (-0.8,-0.3) .. controls (-0.3, -0.5) and (0.2,-0.5) .. (0.6,-0.3);
    \draw[Latex-, thick] (1.3,-0.3) .. controls (1.7, -0.5) and (2.2,-0.5) .. (2.6,-0.3);
    \node[anchor=center] at (-3,0) {$\mathcal{N}_0$};
    \node[anchor=center] at (-1,0) {$\mathcal{Q}_k$};
    \node[anchor=center] at (1,0) {$\mathcal{N}_2$};
    \node[anchor=center] at (3,0) {$\mathcal{W}_n$};
    \end{tikzpicture}
    \vspace{8pt}
    \caption{The diagram which depicts the configuration of the covering relations between the h-sets $\mathcal{N}_0$, $\mathcal{N}_2$, $\mathcal{Q}_k$ and $\mathcal{W}_n$.}
    \label{fig:covering_relations_overview}
\end{figure}

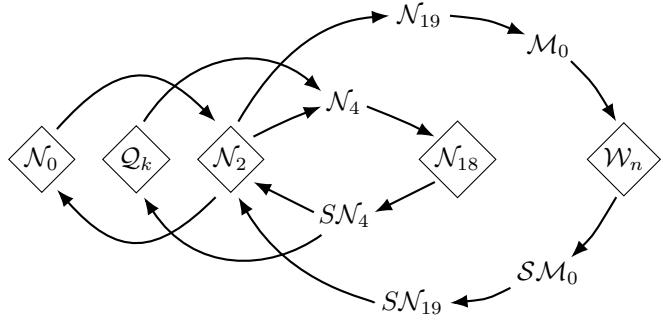
\begin{figure}[ht]
    \centering
    \begin{tikzpicture}
    \draw[-Latex, thick] (-3.8,0.4) .. controls (-3.0,1.3) and (-2.5,1.3) .. (-1.7,0.5);
    \draw[Latex-, thick] (-3.8,-0.4) .. controls (-3.0,-1.3) and (-2.5,-1.3) .. (-1.7,-0.5);
    \draw[-Latex, thick] (-2.65,0.5) .. controls (-2.0,1.5) and (-1.0,1.5) .. (-0.3,0.9);
    \draw[Latex-, thick] (-2.65,-0.5) .. controls (-2.0,-1.5) and (-1.0,-1.5) .. (-0.3,-1.0);
    \draw[-Latex, thick] (-1.2,0.3) .. controls (-0.8, 0.5) .. (-0.3,0.7);
    \draw[Latex-, thick] (-1.2,-0.3) .. controls (-0.8, -0.5) .. (-0.4,-0.7);
    \draw[-Latex, thick] (0.3,0.7) .. controls (0.8, 0.5) .. (1.2,0.3);
    \draw[Latex-, thick] (0.4,-0.7) .. controls (0.8, -0.5) .. (1.2,-0.3);
    \draw[-Latex, thick] (3.0,1.3) .. controls (3.3,1.0) .. (3.6,0.5);
    \draw[Latex-, thick] (3.0,-1.3) .. controls (3.3,-1.0) .. (3.6,-0.5);
    \draw[-Latex, thick] (1.4,1.9) .. controls (1.9,1.85) .. (2.4,1.65);
    \draw[Latex-, thick] (1.4,-1.9) .. controls (1.9,-1.85) .. (2.2,-1.7);
    \draw[-Latex, thick] (-1.4,0.5) .. controls (-0.7,1.7) and (0.0,1.85)  .. (0.6,1.9);
    \draw[Latex-, thick] (-1.4,-0.5) .. controls (-0.7,-1.7) and (0.4,-1.85)  .. (0.4,-1.9);
    \node[draw, diamond, aspect=1, inner sep=1pt] at (-4.0,0) {$\mathcal{N}_0$};
    \node[draw, diamond, aspect=1, inner sep=1pt] at (-2.75,0) {$\mathcal{Q}_k$};
    \node[draw, diamond, aspect=1, inner sep=1pt] at (-1.5,0) {$\mathcal{N}_2$};
    \node[anchor=center] at (0,0.75) {$\mathcal{N}_4$};
    \node[anchor=center] at (0,-0.75) {$S\mathcal{N}_4$};
    \node[anchor=center] at (1.0,1.9) {$\mathcal{N}_{19}$};
    \node[anchor=center] at (0.9,-1.9) {$S\mathcal{N}_{19}$};
    \node[draw, diamond, aspect=1, inner sep=1pt] at (1.5,0) {$\mathcal{N}_{18}$};
    \node[anchor=center] at (2.7,1.5) {$\mathcal{M}_0$};
    \node[anchor=center] at (2.7,-1.5) {$\mathcal{SM}_0$};
    \node[draw, diamond, aspect=1, inner sep=1pt] at (3.7,0) {$\mathcal{W}_n$};
    \end{tikzpicture}
    \vspace{8pt}
    \caption{Extension of the diagram from the figure \ref{fig:covering_relations_overview} including the auxiliary h-sets which are present in our construction. Symmetric sets are marked with diamonds.}
    \label{fig:covering_relations_overview_detailed}
\end{figure}

The information about where to locate our h-sets is obtained by computing the coordinates of the points at which the stable and unstable manifolds of the relevant invariant sets intersect. This calculation is performed with the nonrigorous numerical computations, but it provides sufficient accuracy to enable us to specify the coordinates of the h-sets which are then used in the rigorous computations.

This work is the direct extension of the results published in \cite{CAPINSKI2026109173}, where the analogous proof was provided, but only for the existence of the oscillatory motion to collision in the Earth-Moon system.

Our result is related to a number of previous works concerning collisions, near collision orbits and the oscillatory motion. In 1960 Sitnikov \cite{MR0127389} constructed oscillatory motion to infinity for a particular case of a symmetric restricted spacial three body problem we nowadys refer to as the Sitnikov model. His results were later improved by Alexeev \cite{VMAlekseev_1976} and Moser \cite{Moser01}, who used the intersections of stable/unstable manifolds to infinity.

In the planar setting, the first result was obtained by Llibre and Sim\'o \cite{SimoL80} who proved the existence of the oscillatory motion to infinity for the sufficiently small mass ratio $\mu$. Their result was extended by Xia \cite{XIA1992170} for the range of all $\mu \in (0,1/2]$ except for the finite set of values and then extended further by Guardia, Martin and Seara \cite{GuardiaMS16} for the complete range of $\mu \in (0,1/2]$. We notice that all these results \cite{SimoL80,XIA1992170,GuardiaMS16} assume that the value of the Jacobi constant is sufficiently large.

Oscillatory motion to infinity in the setup with the realistic mass ratio and energy level parameters was achieved with the computer-assisted proof by Capiński, Guardia, Martin, Seara and Zgliczyński \cite{MR4391693}. The authors considered the planar circular restricted three body problem as an approximation of the Sun-Jupiter system with the Jacobi constant allowing to cross the Jupiter's orbit.

In the context of the evolution to collision in the planar circular restricted three body problem, the first results were presented by Llibre and Lacomba \cite{llibre1982,LACOMBA198869} who proved the existence of the transversal ejection/collision orbits for the sufficiently small mass ratio $\mu$ and sufficiently large value of the Jacobi constant.

The consecutive close encounters with the small primary were the subject of the study by Font, Nunes, and Simó \cite{MR1877971,MR2475705}, who proved the existence of the chaotic orbits, which approach the collision infinitely many times. The mass ratio $\mu$ is treated as the perturbation parameter and with the use of numerical computations, the authors prove the applicability of their method for $\mu$ up to $10^{-3}$.

In \cite{MR0967629} Chenciner and Llibre used KAM arguments combined with the Levi-Civita coordinates to prove that the collision circle intersects transversally with an invariant tori. Their results were later extended in \cite{OLLE2018298,OLLE2020105294,seara2022}, where the authors introduced and explored the concept of the $n$-ejection-collision orbits.

Regarding the target to combine all types of motion we mention the work of Moeckel \cite{MR2350333} who has proven the existence of oscillatory motions and close to collision orbits via symbolic dynamics for the general planar three body problem by relying on triple collision approaches.

Significant progress has been achieved recently with the results from Guardia, Lamas and Seara \cite{MR4890942}, who proved that all types of motion (hyperbolic, parabolic, oscillatory to collision and oscillatory to inifinity) in the planar circular restricted three body problem can be combined. The proof is computer-assisted and it assumes sufficiently small mass ratio $\mu$ and the energy value sufficiently close to zero.

Our work is organized as follows. In section \ref{sec:preliminaries} we cover the preliminaries, which includes the definition of the planar circular restricted three body problem, the Levi-Civita coordinates, the McGehee coordinates, interval Newton method and the covering relations. We also recall the key constructions and results from \cite{CAPINSKI2026109173} which sets up the notation and serves as a baseline for the new developments. In the subsection \ref{subsec:poincare_sections} we provide the description of how we set up Poincar\'e sections and covering relations on a constant energy level with an emphasis on the utilization of the time reversal symmetry of the system. Then, we discuss the properties of the ejection/collision orbit, its homoclinic and the related covering relations in the subsection \ref{subsec:EC_orbit_and_homoclinic}. In section \ref{sec:main_result} we present the main results. We discuss the construction of the covering relations located along the heteroclinic orbit from collision to infinity in subsection \ref{subsec:covrel_to_infinity}. We construct the covering relations in the neighborhood of infinity in subsection \ref{subsec:covrel_at_infinity}. We conclude our work with the proof of Theorem \ref{thm:primary_result} in subsection \ref{subsec:main_proof}.

\section{Preliminaries}
\label{sec:preliminaries}
\subsection{The planar circular restricted 3-body problem\label{subsec:PCR3BP}}
In this work we consider the same setup as in \cite{CAPINSKI2026109173}. The planar circular restricted three body problem (PCR3BP) is a celestial mechanics model where a massless test particle moves under a gravitational influence of the two point masses called primaries, which are moving along the circular orbits around the center of the mass of the system. 

We select the coordinate system whose origin coindices with the center of mass of the primaries and which is co-rotating with them, so that the masses become stationary at points $(x_1, 0)$ and $(x_2, 0)$, respectively. We assume the units of time, mass and distance such that $\mu_2  = x_1$, $\mu_1 = -x_2$ and $\mu_1 + \mu_2 = 1$. The Hamiltonian which describes the dynamics of the test particle in such setting takes form:
\begin{equation}
    H(x,y,p_{x},p_{y})=\frac{1}{2}(p_{x}^{2}+p_{y}^{2})+yp_{x}-xp_{y}-\sum_{i\in\{1,2\}}\mu_{i}\big((x-x_{i})^{2}+y^{2}\big)^{-\frac{1}{2}},
\end{equation}
where $x$,$y$ are the position coordinates of the test particle and $p_{x}$,$p_{y}$ are their associated momenta.

Movement of the test particle is determined by the Hamilton equation:
\begin{equation}
    \label{eq:std_ode}
    \frac{d}{dt}(x,y,p_{x},p_{y})=J\nabla H(x,y,p_{x},p_{y}),
\end{equation}
where $t$ is time and 
\begin{equation}
    \label{eq:J_matrix_def}
    J=\begin{bmatrix}0 & I_{2}\\-I_{2} & 0 \end{bmatrix},\quad
    I_{2}=\begin{bmatrix}1 & 0\\0 & 1\end{bmatrix}.
\end{equation}
We denote the flow induced by the differential equation (\ref{eq:std_ode}) by $\Phi_t$ and we notice that the system features a time reversal symmetry with respect to the function $S$ which is defined by:
\begin{equation}
    \label{eq:S_symmetry}
    S(q_1, q_2, q_3, q_4) = (q_1, -q_2, -q_3, q_4).
\end{equation}
We recall that the time reversal symmetry is defined by the relationship:
\[
    \Phi_{-t} (S(q)) = S( \Phi_t (q) ) \quad\text{for}\quad (t,q) \in \dom \Phi.
\]

The Hamiltonian has singularities for $(x,y)$ equal to $(x_1, 0)$ or $(x_2, 0)$. We interpret them as the collisions with the first or the second primary mass respectively. In order to be able to compute the trajectories which pass through these singularities, it is necessary to regularize the system. We achieve it by applying the Levi-Civita coordinates \cite{levi_civita_1906} which consists in the suitable phase space coordinate change and the time coordinate change.

We introduce a coordinate change $(x,y,p_x,p_y)=\gamma_i (u,v,p_u,p_v)$ such that:
\begin{equation}
    \label{eq:lc_coord_change}
    \begin{aligned}
        x & =u^{2}-v^{2}+x_{i},\\
        y & =2uv,\\
        p_{x} & =\frac{1}{2}\frac{up_{u}-vp_{v}}{u^{2}+v^{2}},\\
        p_{y} & =\frac{1}{2}\frac{vp_{u}+up_{v}}{u^{2}+v^{2}},
    \end{aligned}
\end{equation}
where $i \in \{ 1, 2 \}$ is fixed and we introduce new time coordinate $s$ such that:
\begin{equation}
    \label{eq:reg_time_def}
    \frac{dt}{ds}(s)=4\big(u^{2}(s)+v^{2}(s)\big).
\end{equation}
The value $i$ specifies which primary mass we regularize (the trajectories which pass through the collision with this mass shall no longer have the singularity).

We define Hamiltonian $\Gamma_{i,h}$ with $h$ being a fixed real parameter, by
\begin{equation}
    \label{eq:Gamma_i}
    \Gamma_{i,h} (u,v,p_u,p_v) = 4(u^2 + v^2) \left( H(\gamma_i(u,v,p_u,p_v)) - h \right).
\end{equation}

Movement of the test particle in coordinates $(u,v,p_u,p_v)$ with respect to the time coordinate $s$ is determined by the Hamilton equation:
\begin{equation}
    \label{eq:reg_ode}
    \frac{d}{ds}(u,v,p_{u},p_{v})=J\nabla \Gamma_{i,h} (u,v,p_{u},p_{v}),
\end{equation}
under the assumption that
\begin{equation}
    \label{eq:gamma_zero_condition}
    \Gamma_{i,h}(w_0) = 0,\quad \text{(or equivalently }h = H(\gamma_i(w_0))),
\end{equation}
where $w_0 = (u_0, v_0, p_{u,0}, p_{v,0})$ is the initial condition of the equation (\ref{eq:reg_ode}). The value $h$ can be interpreted as the energy level for which we perform the regularization.

We denote the flow induced by the differential equation (\ref{eq:reg_ode}) by $\Phi_s^{i,h}$ and we notice that it also has the time reversal symmetry with respect to the function $S$ defined in (\ref{eq:S_symmetry}).

The relationship between the flow $\Phi_s^{i,h}$ and the flow $\Phi_t$ is:
\[
    \Phi_{s}^{i,h} (w) = \Phi_{t(s)} (\gamma_i(w)),
\]
where $t(s)$ is integrated according to the formula (\ref{eq:reg_time_def}), with the initial condition $t(0) = 0$.

The position of the regularized mass (specified by the index $i$) in the new coordinates is $(u,v) = (0,0)$. If we combine this fact with the condition (\ref{eq:gamma_zero_condition}) we reach the conclusion, that the collision with the regularized mass must satisfy:
\begin{equation}
    p_{u}^{2}+p_{v}^{2}=8\mu_{i}, \label{eq:collision-circle}
\end{equation}
regardless of the selected energy level $h$.

\subsection{McGehee coordinates}
In this section we present the coordinate change of the planar circular restricted three body problem, which allows for the numerical computation of the trajectories, which are at the arbitrarily large distance from the barycentre of the system, including the limit case (i.e. infinite distance). Originally, these coordinates have been introduced in \cite{MCGEHEE197370}. The transformation presented in this chapter includes an additional step (we follow the reasoning from \cite{CAPINSKI2022316}), which provides a good alignment of the coordinate directions with the dynamics of the system.

The first step is to express the Hamiltonian $H$ in the polar coordinates $(q_r,\theta,p_r,p_\theta)$, where
\begin{equation*}
    \begin{split}
        x &= q_r \cos(\theta), \\
        y &= q_r \sin(\theta), \\
        p_x &= \cos(\theta) p_r - p_\theta \sin(\theta) / q_r, \\
        p_y &= \sin(\theta) p_r + p_\theta \cos(\theta) / q_r. \\
    \end{split}
\end{equation*}
Coordinate change is canonical and the Hamiltonian $H$ takes form:
\[
    H = \frac{1}{2} \left(p_r^2 + p_\theta^2/q_r^2 \right) - p_\theta - \sum_{k=1,2} \mu_k \left( q_r^2 - 2x_k q_r \cos(\theta) + x_k^2 \right)^{-1/2}.
\]
The next step is to replace the coordinate $q_r$, with the coordinate $\chi$, which we define by the equation $q_r = 2 \chi^{-2}$. The equations of motion induced by the Hamiltonian $H$, in the coordinates $(\chi, \theta, p_r, p_\theta)$, which we call McGehee coordinates, take form: 
\begin{equation}
    \label{eq:vector_field_mcgehee_short}
    \begin{split}
        \dot{\chi} &= -\frac{1}{4} \chi^3 p_r, \\
        \dot{\theta} &= \frac{1}{4} \chi^4 p_\theta - 1, \\
        \dot{p}_r &=
            -\frac{1}{4} \chi^4 -\frac{1}{4} \chi^4 f_1(\chi, \theta) + \frac{1}{8} \chi^6 \left( p_\theta^2 + \cos(\theta) f_2(\chi, \theta) \right), \\
        \dot{p}_\theta &= - \frac{1}{4} \chi^4 \sin(\theta) f_2 (\chi, \theta), \\
    \end{split}
\end{equation}
where we define functions $f_1$ and $f_2$ by:
\begin{equation}
    1 + f_1 (\chi, \theta) = \sum_{k=1,2} \mu_k \left(1 - x_k \chi^2 \cos(\theta) + \frac{1}{4} x_k^2 \chi^4 \right)^{-3/2},
\end{equation}
\begin{equation}
    f_2 (\chi, \theta) = \sum_{k=1,2} \mu_k x_k \left(1 - x_k \chi^2 \cos(\theta) + \frac{1}{4} x_k^2 \chi^4 \right)^{-3/2}.
\end{equation}
\begin{remark}
    Functions $f_1$ and $f_2$ are $2\pi$-periodic with respect to the coordinate $\theta$. On the other hand, due to the equation
    \[
        \left(1 - x_k \cos(\theta) \chi^2 + \frac{1}{4} x_k^2 \chi^4 \right)^{-3/2} = 1 + \frac{3}{2} x_k \cos(\theta) \chi^2 + O(\chi^4)
    \]
    (which is the direct consequence of the Taylor expansion with respect to the variable $\chi^2$) and
    \[
        \mu_1 x_1 + \mu_2 x_2 = 0,
    \]
    we can determine the asymptotic behavior of both functions for $\chi$ going to $0$:
    \[
        f_1 (\chi,\theta) = O(\chi^4) \quad\text{and}\quad
        f_2 (\chi,\theta) = \frac{3}{2} \cos(\theta) \mu_1 \mu_2 \chi^2 + O(\chi^4).
    \]
\end{remark}

The function $H$ expressed in the McGehee coordinates takes form:
\begin{equation}
\label{eq:Hamiltonian_in_McGehee_coordinates}
    H = \frac{1}{8} p_\theta^2 \chi^4 - p_\theta + R_H \left( \chi, \theta, p_r \right),
\end{equation}
where
\[
    R_H \left( \chi, \theta, p_r \right) = \frac{1}{2} p_r^2 - \sum_{k=1,2} \mu_k \chi^2 \left( 4 - 4 \chi^2 x_k  \cos(\theta) + \chi^4 x_k^2 \right)^{-1/2}.
\]

Due to the fact that the coordinate change from $(q_r, \theta, p_r, p_\theta)$ to $(\chi, \theta, p_r, p_\theta)$ is not canonical, the function $H$ does not give rise to the system of ordinary differential equations (\ref{eq:vector_field_mcgehee_short}) according to the Hamilton equation of motion.

On the other hand, the value of the function $H$ remains constant along the solution of the system of equations (\ref{eq:vector_field_mcgehee_short}). For the fixed energy level $h_0$ it is possible to compute the coordinate $p_\theta$, according to the formula:
\begin{equation}
    \label{eq:E_H_solution}
    p_\theta = E_H (h_0, \chi, \theta, p_r) = \frac{2\left(R_H \left( \chi, \theta, p_r \right)- h_0 \right)}{ 1 + \sqrt{ 1 - \chi^4 \left(R_H \left( \chi, \theta, p_r \right) - h_0 \right) / 2 } }.
\end{equation}
We shall make use of this fact later in order to parameterize the Poincar\'e section restricted to the constant energy level in the McGehee coordinates.

\begin{remark}
    \label{rem:E_H_even}
    The function $E_H$ is even with respect to the variable $p_r$, i.e.
    \[
        E_H (h_0, \chi, \theta, p_r) = E_H (h_0, \chi, \theta, -p_r).
    \]
\end{remark}

Similarily as in the case of the Levi-Civita coordinates, each point in the original coordinates corresponds to two points in the McGehee coordinates. The equation (\ref{eq:vector_field_mcgehee_short}) features time-reversal symmetry with respect to the function $S$ defined in (\ref{eq:S_symmetry}).

For each pair of values $p_r, p_\theta \in \mathbb{R}$, the set $\{ (\chi, \theta, p_r, p_\theta) : \chi = 0, \theta \in [0, 2\pi) \}$ is an invariant set for the dynamical system induced by the system of the differential equations (\ref{eq:vector_field_mcgehee_short}).

Stable and unstable sets of these invariants contain the parabolic (for $p_r = 0$) and hyperbolic (for $p_r \neq 0$) orbits of the planar circular restricted three body problem (see Figure \ref{fig:mcgehee_dynamics} for the illustration of the dynamics in the neighborhood of the parabolic limit point). This is implied by the fact that the radial velocity $\dot{q}_r$ expressed in the McGehee coordinates takes form $\dot{q}_r = -4\chi^{-3} \dot{\chi} = p_r$, which means, that the investigation of the sign of coordinate $p_r$ is sufficient to determine the classification of the orbit which goes to infinity.

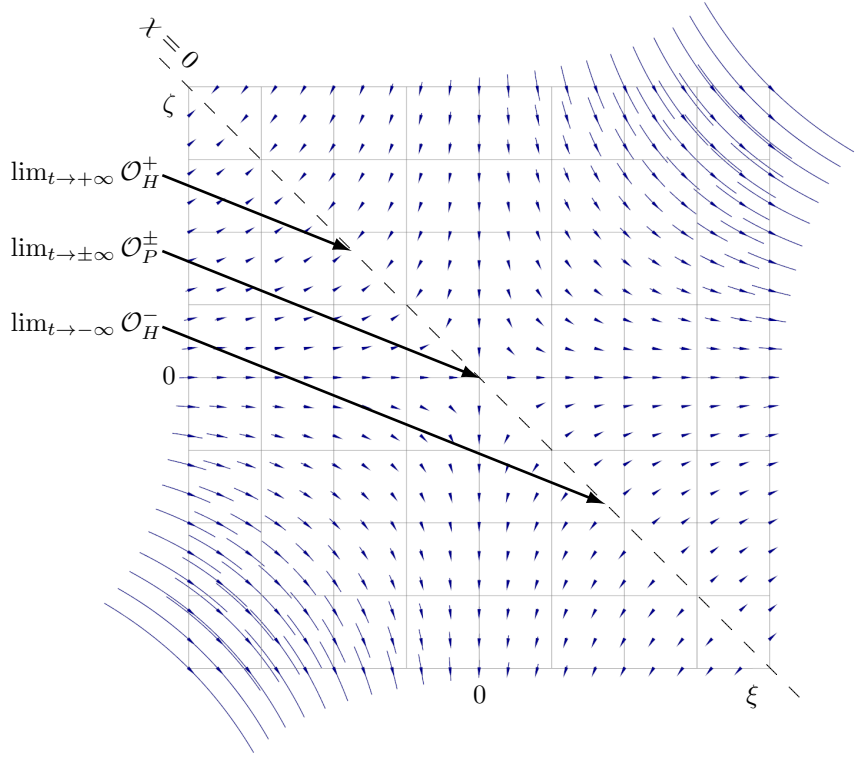
\begin{figure}[h!]
\begin{center}
	\begin{overpic}[width=10cm]{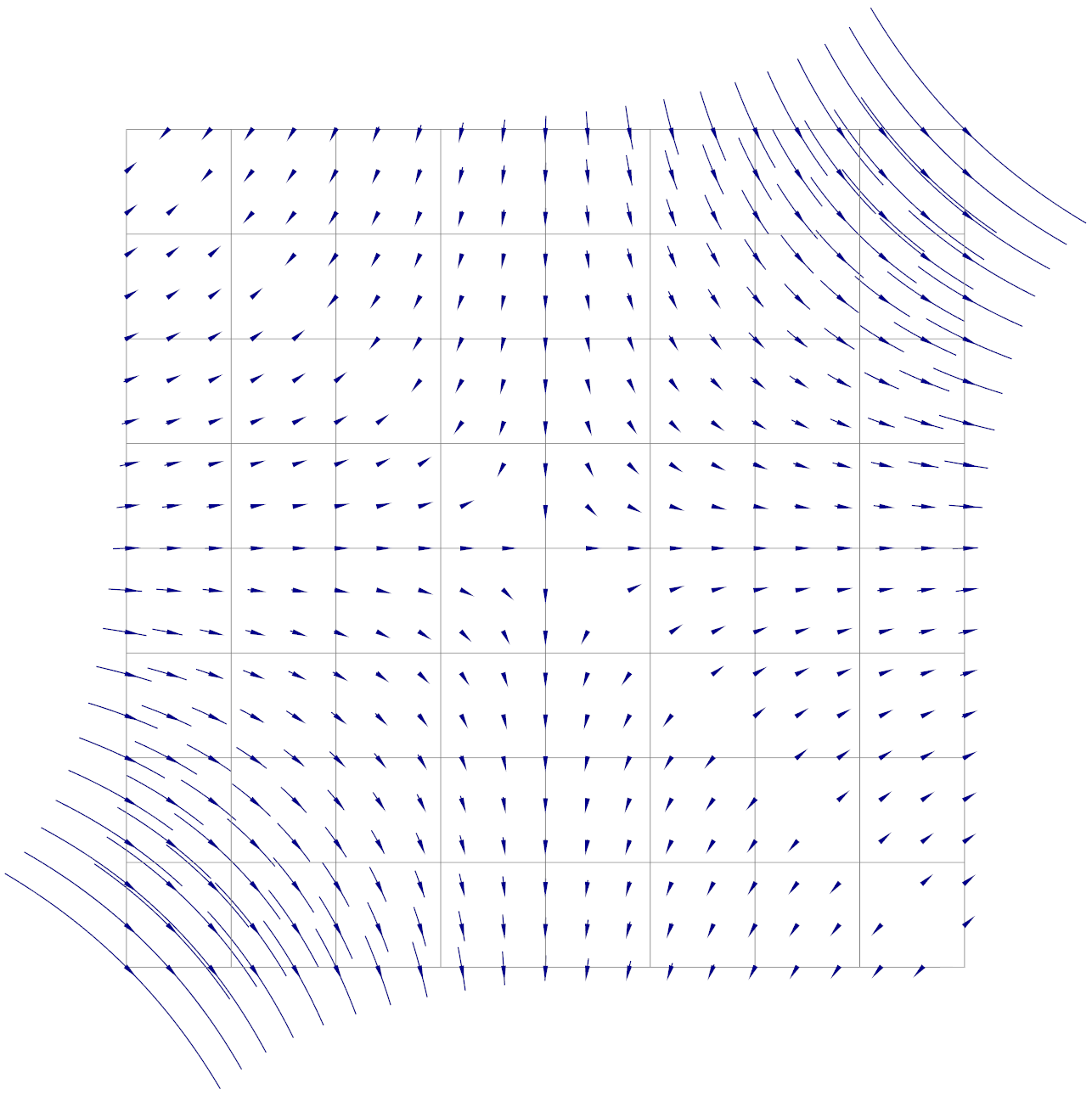}
        \dashline{2}(7.7,92)(92,7.8)
        \put (8,85) {$\zeta$}
		\put (49,7) {$0$}
		\put (-12,76) {$\lim_{t \to +\infty} \mathcal{O}_H^+$}
		\put (-12,66) {$\lim_{t \to \pm\infty} \mathcal{O}_P^\pm$}
		\put (-12,56) {$\lim_{t \to -\infty} \mathcal{O}_H^-$}
		\put (8,49) {$0$}
		\put (85,7) {$\xi$}
		\put (9,94) {\makebox(0,0){\rotatebox{-45}{$\chi = 0$}}}

        \linethickness{1pt}

        \put(8,76.6){\vector(2.5,-1){25}}

        \put(8,66.6){\vector(2.5,-1){42}}

        \put(8,56.6){\vector(2.5,-1){58.5}}

	\end{overpic}
	\caption{Dynamics of the PCR3BP in the coordinates $(\xi, \zeta, \theta, p_\theta)$ obtained with the use of approximate numerical computations. The origin of the coordinate system corresponds to the invariant set, for which the stable and unstable sets contain the parabolic orbits ($\mathcal{O}_P^\pm$). Intersection of the invariant manifold $\chi = 0$ (the diagonal marked with the dashed line) and the second and the fourth quarters of the coordinate system corresponds with the limit sets of the hyperbolic orbits $\mathcal{O}_H^+$ and $\mathcal{O}_H^-$ respectively.}
    \label{fig:mcgehee_dynamics}
\end{center}
\end{figure}

We introduce an additional coordinate change from $(\chi, \theta, p_r, p_\theta)$ to the coordinates $(\xi, \zeta, \theta, p_\theta)$ which are defined by the relationships:
\[
    \chi = \frac{1}{2} (\zeta + \xi), \quad p_r = \frac{1}{2} (\zeta - \xi).
\]
Such operation leads to the setup where the suitable projections of the stable and unstable directions of the manifold $(\xi,\eta) = (0,0)$ coincide approximately with the vectors $(1,0,0,0)$ and $(0,1,0,0)$ of the introduced coordinate system.

The vector field (\ref{eq:vector_field_mcgehee_short}) expressed in the coordinate system $(\xi, \zeta, \theta, p_\theta)$ takes form (we skip the arguments of functions  $f_2$ and $f_3$ to simplify the notation):
\begin{equation}
    \label{eq:vector_field_aligned}
    \begin{split}
        \dot{\xi} &= \frac{1}{32} (\zeta + \xi)^3 \left( \xi + \frac{1}{2} (\zeta + \xi) {f_3} \right), \\
        \dot{\zeta} &= -\frac{1}{32} (\zeta + \xi)^3 \left( \zeta + \frac{1}{2} (\zeta + \xi) {f_3} \right), \\
        \dot{\theta} &= \frac{1}{64} (\zeta + \xi)^4 p_\theta - 1, \\
        \dot{p}_\theta &= - \frac{1}{64} (\zeta + \xi)^4 \sin(\theta) f_2, \\
    \end{split}
\end{equation}
where the function $f_3$ is defined by:
\begin{equation}
    f_3 (\chi, \theta, p_\theta) = f_1(\chi, \theta) - \frac{1}{2} \chi^2 \left( p_\theta^2 + \cos(\theta) f_2(\chi, \theta) \right).
\end{equation}
\begin{remark}
    We notice that $f_3 = O (\chi^2)$ 
\end{remark}

We denote the coordinate change which is the composition of all three transformations described in this section as $\hat{\gamma}$, i.e.:
\begin{equation}
    \label{eq:hat_gamma}
    (\xi, \zeta, \theta, p_\theta) = \hat{\gamma} (x,y,p_x,p_y).
\end{equation}

The planar circular restricted three body problem expressed in the coordinates $(\xi, \zeta, \theta, p_\theta)$ has a time reversal symmetry with respect to function $S'$ defined by:
\begin{equation}
    \label{eq:Sp_symmetry}
    S'(x_1, x_2, x_3, x_4) = (x_2, x_1, -x_3, x_4).
\end{equation}

\subsection{Interval arithmetics}

Our work depends heavily on the rigorous numerical computations which are the centerpiece of the computer assisted proofs. One technique which we use extensively is the interval arithmetic. We perform our numerical computations with the use of the CAPD library \cite{MR4283203,MR4395996}, which implements all the necessary operations such as the rigorous evaluation of the Poincar\'e maps.

The other tool which is essential for us is the interval variant of the Newton's method which we use to get the bounds of the solutions of our nonlinear equations. We provide the description of this tool in this section. Let us first set up the needed notation.

We shall refer to a product of closed intervals in $\mathbb{R}^{n}$ as an interval vector. We define an interval enclosure of a set $X\subset\mathbb{R}^{n}$ by:
\[
    [X]=Y_{1}\times...\times Y_{n}\text{ such that }\pi_{i}(X)\subset Y_{i}\text{ for }i\in\{1,...,n\},
\]
where $Y_{1},...,Y_{n}$ are closed real intervals and $\pi_{i}$ is the projection onto the $i$-th component. Note that an interval enclosure is not unique. (We prefer tighter enclosures, since they provide better estimates.) For a set $X\subset\mathbb{R}^{n}$ and for a map $F=(F_{1},...,F_{n}):\mathbb{R}^{n}\to\mathbb{R}^{n}$ we denote an interval enclosure of $F(X)$ as
\[
    \big[F(X)\big]=Y_{1}\times...\times Y_{n}\text{ such that }F_{i}(X)\subset Y_{i}\text{ for }i\in\{1,...,n\}.
\]

A set $A\subset\mathbb{R}^{n\times n}$ is an interval matrix if each of its entries is a closed interval. We say that an interval matrix $A$ is invertible if for every $B\in A$ the matrix $B$ is invertible. We will call an interval matrix $[A^{-1}]$ an interval enclosure of the inverse of $A$ if for every $B\in A$ we have $B^{-1}\in[A^{-1}]$. Note that as in case of interval enclosures in general, the interval enclosure of the inverse of $A$ is not unique, since after it is
enlarged it remains an enclosure of the inverse.

If the map $F$ is $C^{1}$ then we will write $[DF(X)]\subset\mathbb{R}^{n\times n}$ for an interval matrix that satisfies 
\[
    DF(X)\subset[DF(X)].
\]

\begin{theorem}
    \cite{ALEFELD2000421}
    Let $X$ be an interval vector with nonempty interior, let $x\in\interior X$ be some point and let $F:\mathbb{R}^{n}\to\mathbb{R}^{n}$ be $C^{1}$. If for 
    \[
        N(x,X):=x-\big[DF(X)^{-1}\big]\big[F(x)\big]
    \]
    we have $N(x,X)\subset X$, then there exists a unique point $x^{*}\in N(x,X)$ such that 
    \[
        F(x^{*})=0.
    \]
\end{theorem}

This result can be extended. Interval Newton method for the validation of the existence of the solution to implicit equations over an explicit domain is formally stated in the following theorem.

\begin{theorem}
    \label{thm:INM}
    \cite{Walawska2019ValidatedNF}
    Let $X \subset \mathbb{R}^n$ and $Z \subset \mathbb{R}^m$ be interval vectors with nonempty interior, let $x \in \interior X$ be some point and let $F : \mathbb{R}^m \times \mathbb{R}^n \to \mathbb{R}^n$ be a $C^1$ map. If for 
    \[
        N(x,X,Z) = x - \left[ D_x F(Z,X) \right]^{-1} \big[ F(Z,x) \big]
    \]
    we have $N(x,X,Z) \subset \interior X$ then there exists a $C^1$ smooth function $g : Z \to X$ such that the set of zeros $\{ (z,x) \in Z \times X : F(x,z) = 0 \}$ coincides with the graph of the function $g$, i.e.:
    \[
        \left\{ (z,x) \in Z \times X : F(x,z) = 0 \right\} = \left\{ (z,g(z)) : z \in Z \right\}.
    \]
\end{theorem}

\subsection{Covering relations}
In this section we recall the results form \cite{ZGLICZYNSKI200432}, which provide the primary topological tools for obtaining our results. In this paper it will be sufficient for us to consider the two dimensional case where both unstable and stable directions of h-sets have dimensions one, so the definitions provided below are the special cases of the respective definitions from \cite{ZGLICZYNSKI200432}.

Let $B = [-1,1]\subset \mathbb{R}$ and $\partial B = \{ -1, 1 \}$.

\begin{definition}
    \label{def:h_set}
    An h-set is a pair $N = \left( |N|, c_N \right)$ where $|N|$ is a compact subset of $\mathbb{R}^2$ and $c_N : \mathbb{R}^2  \to  \mathbb{R}^2$ is a homeomorphism such that:
    \[
        c_N \left( B \times B \right) = |N|.
    \]
    The subset $|N|$ is called a support of an h-set.
\end{definition}

We denote (see Figure \ref{fig:covering}):
\[
    N_c = B \times B, \quad
    N_c^- = \partial B \times B, \quad 
    N_c^+ = B \times \partial B,
\]
\[
    N^- = c_N ( N_c^- ), \quad
    N^+ = c_N ( N_c^+ ),
\]
\[
    N^{l} = c_N \big( \{ -1 \} \times B \big), \quad
    N^{r} = c_N \big( \{ +1 \} \times B \big).
\]
We shall refer to $N^-$ as the exit set.

In this paper the supports of all h-sets we consider will be quadrangles. We will frequently refer to the support as the h-set and neglect to write out the homeomorphism $c_N$, which simply maps the square $N_c$ into a quadrangle, whenever it will be clear from the context which edges of the quadrangle represent the exit set.

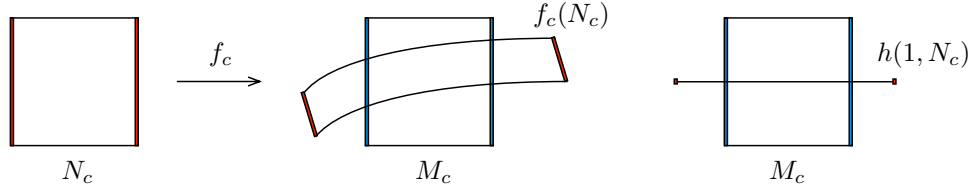
\begin{figure}
\begin{center}
	\begin{overpic}[width=13cm]{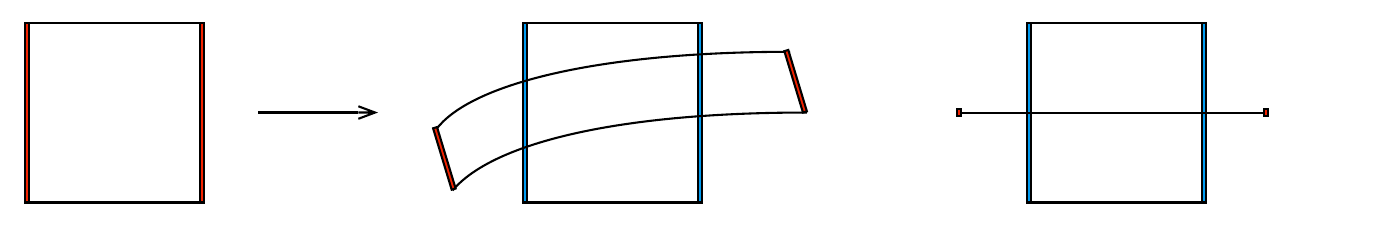}
	\put (7,0){$N_c$}
	\put (43,0){$M_c$}
	\put (79,0){$M_c$}
	\put (22,12){$f_c$}
	\put (55,16){$f_c(N_c)$}
	\put (90,12){$h(1,N_c)$}
	\end{overpic}
\end{center}
\caption{A depiction of a covering relation $N \xRightarrow{f} M$. The exit sets $N_c^-$ and $M_c^-$ are in red and blue, respectively. On the right we have the final result of the homotopy from Definition \ref{def:covering-relation}.\label{fig:covering}}
\end{figure}

\begin{definition}
    \label{def:covering-relation}
    Assume $N$, $M$ are h-sets. Let $f : N \to \mathbb{R}^n$ be a continuous map, such that the map $f_c = c_M^{-1} \circ f \circ c_N : N_c \to \mathbb{R}^2$ is well-defined and continuous. We say that
    \[
        N \xRightarrow{f} M
    \]
    ($N$ $f$-covers $M$) if and only if the following conditions are satisfied (see Figure \ref{fig:covering}):
    
    \begin{enumerate}
        \item There exists a continuous homotopy $h : [0;1] \times N_c \to \mathbb{R}^2$, such that the following conditions hold true:
        \begin{align}
            \begin{aligned}
                h_0 &= f_c, \\
                h \big( [0;1], N_c^- \big) \cap M_c &= \emptyset, \\
                h \big( [0;1], N_c \big) \cap M_c^+ &= \emptyset. \\
            \end{aligned}
        \end{align}
        \item There exists a real number $a > 1$, such that:
        \[
            h(1, (p, q)) = (ap, 0), \quad p,q \in B.
        \]
    \end{enumerate}
\end{definition}

In our computer assisted proofs, we assert the existence of the covering relations by checking the following conditions ($\pi_1$ and $\pi_2$ denote the projection onto the first and the second component respectively):
\begin{align}
    \label{eq:covering_relation_conditions}
    \begin{aligned}
        \pi_{2}\big(f_c(N_c)\big) & \subset B,\\
        \pi_{1}\big(f_c(N_c^{l})\big) & < -1,\\
        \pi_{1}\big(f_c(N_c^{r})\big) & > +1,
    \end{aligned}
\end{align}
which we refer to as the contraction condition, the left expansion condition and the right expansion condition respectively.

\begin{definition}\label{def:NT}
    Let $N$ be an h-set. We define the h-set $N^T$ such that the support of $N^T$ is equal to the support of $N$ and  the homeomorphism $c_{N^T}$ is such that:
    \[
        c_{N^T} = c_N \circ j,
    \]
    where $j : \mathbb{R}^2 \to \mathbb{R}^2$ is given by $j(p,q) = (q,p)$.
\end{definition}

\begin{definition}\label{def:back-covering-relation}
    Assume $N$, $M$ are h-sets. Let $g$ be a function such that $g^{-1} : M \to \mathbb{R}^2$ is well defined and continuous. We say that
    \[
        N \xLeftarrow{g} M
    \]
    ($N$ $g$-backcovers $M$) if and only if $M^T \xRightarrow{g^{-1}} N^T$.
\end{definition}

The following theorem will be our main tool for constructing our oscillating motions.
\begin{theorem}
    \label{thm:generic_covering_relations}
    \cite{ZGLICZYNSKI200432} Let $k \in \mathbb{N}$, $N_0, ..., N_{k-1}$ be a sequence of h-sets, let $N_k=N_0,$ and $f_{1}, \dots, f_{k}$ be a sequence of functions such that
    \[
    	N_{i-1} \xLeftarrow{f_i} N_{i} \qquad \mbox{or} \qquad N_{i-1} \xRightarrow{f_i} N_{i} \qquad \mbox{for }i\in\{1,\ldots,k\}.
    \]
    Then, there exists a point $x$ in the interior of $N_0$ such that:
    \[
        f_{i}\circ f_{i-1}\circ\cdots\circ f_{1}(x) \in N_i, \qquad \mbox{for }i=1,\ldots,k-1,
    \]
    and
    \[
        f_{k}\circ f_{k-1}\circ \cdots \circ f_{1}(x)=x.
    \]
\end{theorem}

The theorem above can be extended (we will use both versions in our arguments) by introducing the following two notions.
\begin{definition}\label{def:horizontal-disc} Let $N$ be an h-set. Let $b:[-1,1]\to |N|$ be continuous and let $b_c=c_N^{-1}\circ b$. We say that $b$ is a {\em horizontal disc} in $N$ if there is a homotopy $h:[0,1]\times[-1,1]\to N_c$ such that
\begin{align*}
	h_0&=b_c, \\
	h_1(x)& = (x,0), \quad \mbox{for all }x\in [-1,1], \\
	h([0,1],x)&\subset N_c^-, \qquad \mbox{for }x\in\{-1,1\}.
\end{align*}
\end{definition}
\begin{definition}\label{def:vertical-disc} Let $N$ be an h-set. Let $b:[-1,1]\to |N|$ be continuous and let $b_c=c_N^{-1}\circ b$. We say that $b$ is a {\em vertical disc} in $N$ if there is a homotopy $h:[0,1]\times[-1,1]\to N_c$ such that
\begin{align*}
	h_0&=b_c, \\
	h_1(y)& = (0,y), \quad \mbox{for all }y\in [-1,1], \\
	h([0,1],y)&\subset N_c^+, \qquad \mbox{for }y\in\{-1,1\}.
\end{align*}
\end{definition}
With these notions we can formulate the following result:
\begin{theorem}\cite{MR2494688}
    \label{thm:cover-discs}
    Let $k \in \mathbb{N}$, $N_0, ..., N_{k-1}$ be a sequence of h-sets and $f_{1}, \dots, f_{k}$ be a sequence of functions such that
    \[
    	N_{i-1} \xLeftarrow{f_i} N_{i} \qquad \mbox{or} \qquad N_{i-1} \xRightarrow{f_i} N_{i} \qquad \mbox{for }i\in\{1,\ldots,k\}.
    \]
    Assume that $b_h$ is a horizontal disc in $N_0$ and $b_v$ is a vertical disc in $N_k$.
    
    Then, there exists a point $x$ in the interior of $N_0$ such that:
    \begin{align*}
    	x&=b_h(t) &\mbox{for some }t\in (-1,1), \\
	    f_{i}\circ f_{i-1}\circ\cdots\circ f_{1}(x) &\in N_i, &\mbox{for }i=1,\ldots,k-1,\,\,\, \\
	    f_{k}\circ f_{k-1}\circ\cdots\circ f_{1}(x)&=b_v(s) &\mbox{for some }s\in(-1,1).
    \end{align*}
  
\end{theorem}

\subsection{Poincare sections and symmetries}
\label{subsec:poincare_sections}
In order to simplify the analysis of the dynamics of the planar circular restricted three body problem we make use of the method of the Poincar\'e maps \cite{poincaré1890probleme}. This technique relies on the selection of the subsets of the phase space of codimension one called Poincar\'e sections and the analysis of how the point from one section is mapped onto another by following the trajectory of the dynamical system.

The technique is general and it can be used to transform an arbitrary continuous time dynamical system into a discrete time dynamical system with the phase space of a lower dimension. In principle, we expect the analysis and the description of such system to be simpler than the original one.

However, the system we consider has two properties which we can use to simplify its analysis even further. First, the system is defined by the time independent Hamiltonian. Secondly, the system features a time reversal symmetry which we can make use of to reduce the amount of the numerical computations and improve their accuracy.

In this chapter, we present the approach from \cite{CAPINSKI2026109173} which takes all of the mentioned facts into account to construct the Poincar\'e maps between the sections restricted to a constant energy level. We also discuss the relationship between the Poincar\'e maps and the covering relations and the time reversal symmetry of the system.

Let $\phi (t,x) : \mathbb{R} \times \mathbb{R}^{n} \supset \dom\phi \to \mathbb{R}^{n}$ be the flow of the continuous time dynamical system, where the positive integer $n$ is the dimension of the phase space.

We consider the differentiable function $\Sigma: \mathbb{R}^{n} \supset \dom\Sigma \to \mathbb{R}$ such that the flow of the dynamical system $\phi$ is transversal to the set $\{ \Sigma = 0 \}$, that is:
\[
    D\Sigma (x) \cdot D_t \phi(0,x) \neq 0
\]
for all $x \in \mathbb{R}^n$ which satisfy $(0,x) \in \dom \phi$ and $\Sigma(x) = 0$ ($D\Sigma$ denotes the derivative of $\Sigma$ and it is considered to be a row vector, $D_t \phi$ denotes the partial derivative of $\phi$ with respect to time $t$ and it is interpreted as a column vector). The set $\{ \Sigma = 0 \}$ is called the Poincar\'e section (for simplicity we shall refer to it as section $\Sigma$).

We define the function $T_{\Sigma} : \mathbb{R}^{n} \supset \dom T_\Sigma \to \mathbb{R}_+$ by:
\[
    T_\Sigma (x) = \inf \left\{ t > 0: \Sigma \left( \phi(t,x) \right) = 0 \right\},
\]
where the domain of the function $T_\Sigma$ is the set of $x\in\mathbb{R}^{n}$ such that the set
\[
    \left\{ t > 0: \Sigma \left( \phi(t,x) \right) = 0 \right\}
\]
is not empty. The function $T_\Sigma$ is called the time to section $\Sigma$.

We define the Poincar\'e map onto section $\Sigma$ by
\[
    P_\Sigma (x) = \phi( T_\Sigma(x), x),
\]
where the domain of $P_\Sigma$ is equal the domain of $T_\Sigma$.

Let $\Sigma_1$, $\Sigma_2$ be the two Poincar\'e sections. We define the Poincar\'e map from section $\Sigma_1$ onto section $\Sigma_2$ by $P_{\Sigma_2, \Sigma_1} = P_{\Sigma_2}$, where the domain of $P_{\Sigma_2, \Sigma_1}$ is restricted to section $\Sigma_1$, i.e.:
\[
    \dom P_{\Sigma_2, \Sigma_1} = \dom P_{\Sigma_2} \cap \{ \Sigma_1 = 0 \}.
\]

We recall that our dynamical system is defined by the time independent Hamiltonian $\Gamma$, which means that value of $\Gamma$ remains constant along the trajectory of the system. We can make use of this fact to simplify the analysis of our setup even further, by lowering the dimension of the domain and image of the Poincar\'e maps between sections by one. Instead of analyzing the dynamics of the Poincar\'e maps from section to section, we shall analyze the dynamics of the Poincar\'e maps between an intersections of the sections and the zero energy level set $\{ \Gamma = 0 \}$ \footnote{Hamiltonian is invariant under the addition of a constant, which means that restriction to the zero energy level can be transformed into the restriction to an arbitrary energy level in a straighforward way.} (we choose the energy level equal zero for the compatibility with the condition (\ref{eq:gamma_zero_condition})).

Let us discuss how we select the Poincar\'e sections and how we parameterize their restrictions to the constant energy level.

In our constructions, every Poincar\'e section is a hyperplane of codimension one. In order to specify it uniquely, it is sufficient to specify an arbitrary point, which belongs to it and the normal vector.

We shall always specify the point $w \in \mathbb{R}^n$, which shall be located in the vicinity of some relevant orbit (periodic orbit, homoclinic orbit, heteroclinic orbit, etc.). The normal vector $\hat{n}$ shall always be an approximation of the flow of the system at the specified point $w$, i.e. $J \nabla \Gamma(w)$. We emphasize that both quantities $w$ and $\hat{n}$ do not need to be chosen exactly and it is sufficient that $\Gamma(w) \approx 0$, which makes this approach suitable for the numerical computation.

Taking the considerations above into account, we define every Poincar\'e section $\Sigma$ according to the formula:
\begin{equation}
    \label{eq:section-choice}
    \Sigma (p) = \langle \hat{n}, p-w \rangle,
\end{equation}
where $\langle\cdot,\cdot\rangle$ denotes the scalar product.

Let $A \in \mathbb{R}^{n \times n}$ be an invertible matrix. We introduce the shorthand notation $(\Gamma, \Sigma)(x)$ for the function $x \mapsto \left( \Gamma(x), \Sigma(x) \right)$ and we implicitly define $E : \mathbb{R}^{n-2} \to \mathbb{R}^2$ to be the function satisfying
\begin{equation}
    \label{eq:E-implicit-def}
    (\Gamma,\Sigma)(w + A(z,E(z)))=0,
\end{equation}
where $(z, E(z))$ denotes a $n$-dimensional vector which comes from the concatenation of $z$ and $E(z)$.

We notice that for the fixed $w$, fixed $A$ and fixed $z$, the $E(z)$ can be computed in interval arithmetic by means of the interval Newton method (see Theorem \ref{thm:INM}).

Finally, we define the function $\psi : \mathbb{R}^{n-2} \supset \dom\psi \to \mathbb{R}^{n}$ by:
\begin{equation}
    \label{eq:local-parametrization-choice}
    \psi(z) = w + A(z, E(z)),
\end{equation}
and we shall refer to $\psi$ as the local coordinate system of section $\Sigma$ on the zero energy level in the neighborhood of $w$.

An inverse of $\psi$ can be computed in a straighforward way:
\begin{equation}
    \label{eq:inverse-psi}
    \psi^{-1} (p) = \pi_{1, ..., n-2} A^{-1} (p - w)
\end{equation}
under the assumption that $p$ belongs to $\{\Sigma = 0 \} \cap \{ \Gamma = 0 \}$, where $\pi_{1, ..., n-2}$ denotes the projection onto the first $n-2$ coordinates.

Let $\Sigma_1$ and $\Sigma_2$ be two Poincare sections equipped with the local coordinate systems $\psi_1$ and $\psi_2$ constructed according to the description provided before, i.e.:
\[
    \psi_1 (z) = w_1 + A_1 (z, E_1 (z)),
\]
\[
    \psi_2 (z) = w_2 + A_2 (z, E_2 (z)).
\]
The Poincare map from section $\Sigma_1$ onto section $\Sigma_2$ restricted to the zero energy level, expressed in the local coordinate systems takes form:
\[
    P_{21} = (\psi_2)^{-1} \circ P_{\Sigma_2, \Sigma_1} \circ \psi_1 : \mathbb{R}^{n-2} \to \mathbb{R}^{n-2}.
\]
The construction of the local coordinate systems $\psi_1$ and $\psi_2$ according to the description presented above shall provide correct, rigorous results for an arbitrary invertible matrices $A_1$ and $A_2$, as long as the functions $E_1$ and $E_2$ can be successfully evaluated with the use of interval Newton method.

In our case, we choose the matrix $A$ such that
\begin{equation}
    \label{eq:A-form}
    A = \left[ \hat{u} \quad \hat{s} \quad \hat{n} \quad \hat{h} \right],
\end{equation}
where $\hat{u}$ and $\hat{s}$ are the approximations of the unstable and stable directions at point $w$ respectively. The vector $\hat{n}$ is the normal vector of the Poincar\'e section (see (\ref{eq:section-choice})) and the vector $\hat{h}$ is the approximation of $\nabla \Gamma (w)$.

The selection of $\hat{u}$ and $\hat{s}$ leads to the good alignment of the local coordinates with respect to the dynamics of the system: the image of the point $(0,0)$ under local Poincar\'e map will be close to $(0,0)$, and for $z=(z_1, z_2)$, the first coordinate $z_1$ will be expanding for the local map and $z_2$ will be contracting. Furthermore, the derivatives of $P_{21}$ will be close to diagonal.  On the other hand, the selection of $\hat{n}$ and $\hat{h}$ minimizes the numerical errors.

Let us now discuss how we establish the covering relations in the described setup.

\begin{definition}
    \label{def:h_set_on_section}
    Let $\mathcal{N} := (N, \psi)$ be a pair that consists of an h-set $N$ and a parametrization $\psi$ of section $\{\Gamma = 0\} \cap \{ \Sigma = 0 \}$. We shall refer to the object $\mathcal{N}$ as an h-set $N$ on section $\Sigma$.
\end{definition}

We consider $\mathcal{N}_1 = (N_1, \psi_1)$ and $\mathcal{N}_2 = (N_2, \psi_2)$. We will use the notation:
\begin{equation}
    \label{eq:sec-covering-1}
    \mathcal{N}_1 \xRightarrow{P} \mathcal{N}_2 \qquad\text{and}\qquad
    \mathcal{N}_1 \xLeftarrow{P} \mathcal{N}_2
\end{equation}
to denote the covering relations
\begin{equation}
    \label{eq:sec-covering-2}
    N_1 \xRightarrow{P_{21}} N_2 \qquad\text{and}\qquad
    N_1 \xLeftarrow{P_{21}} N_2 
\end{equation} 
respectively, when it is clear from the context with which Poincar\'e sections and with which local coordinates the map $P$ is associated.

Let us assume that the system driven by the Hamiltonian $\Gamma$ features a time reversal symmetry with respect to a function $S : \mathbb{R}^n \to \mathbb{R}^n$, such that $S(x) = \hat{S}x$, where $\hat{S} \in \mathbb{R}^{n \times n}$ is a unitary matrix, i.e. $\hat{S} \hat{S}^T = I$. In this setting, the time reversal symmetry with respect to the function $S$ is equivalent to:
\[
    \Gamma \circ S = \Gamma \quad \text{and}\quad
    \hat{S} \cdot J \cdot \hat{S}^T = -J, \quad\text{where}\quad
    J = \begin{pmatrix} 0 & I_{n/2} \\ -I_{n/2} & 0 \end{pmatrix},
\]
where $I_{n/2}$ is the identity matrix of dimension $n/2 \times n/2$ (our system is driven by the Hamiltonian, so we implicitly assume that the phase space dimension $n$ is an even number).

Let $\Sigma$ be the hyperplane section, which contains the point $w \in \mathbb{R}^n$, and its normal vector is $\hat{n} \in \mathbb{R}^n$. Then, the function $\Sigma S = \Sigma \circ S$ defines the hyperplane section which contains the point $\hat{S}^T w$ and its normal vector is $\hat{S}^T \hat{n}$. For a given $\psi:\mathbb{R}^{n-2}\to \{\Gamma =0\} \cap \{ \Sigma = 0 \}$, $S\psi$ is a parameterization of $\{ \Sigma S = 0\} \cap\left\{\Gamma=0\right\}.$ For $\mathcal{N} = (N,\psi)$ we define:
\[
    S\mathcal{N} := (N^T, S\psi)
\]
and we call it an image of the h-set $\mathcal{N}$ under the symmetry $S$.

\begin{lemma}
    \label{lem:S_backsymmetry}
    Consider  $\mathcal{N}_1 = (N_1, \psi_1)$ and $\mathcal{N}_2 = (N_2, \psi_2)$, where $N_1$ and $N_2$ are two h-sets on the sections $\Sigma_{1}$ and $\Sigma_{2}$, respectively. The sections are parameterized by $\psi_1$ and $\psi_2$, respectively. Consider also two Poincar\'e maps $P:\Sigma_{1}\rightarrow\Sigma_{2}$ and $P^{S}:S\Sigma_{2}\rightarrow S\Sigma_{1}$. If
    \begin{equation}
        \label{eq:Poinc-cover}
        \mathcal{N}_1 \xRightarrow{P} \mathcal{N}_2
    \end{equation}
    then
    \begin{equation}
        \label{eq:symmetry-back-covering}
        S \mathcal{N}_2 \xLeftarrow{P^S} S \mathcal{N}_1.
    \end{equation}
\end{lemma}

\begin{definition}
    \label{def:self-S-symmetry}
    We will say that $\mathcal{N} = (N,\psi)$ is self $S$-symmetric, if
    \[
        S \circ \psi \circ c_{N^T}=\psi \circ c_N
    \]
    where $N$ is an h-set with a homeomorphism $c_N$ (see Definition \ref{def:NT} for the notation $c_{N^T}$).
\end{definition}

\begin{remark}
    \label{rem:S_symmetry_cN}
    If the local coordinate system $\psi$ satisfies the condition $S \circ \psi \circ j = \psi$ (we recall that $j : \mathbb{R}^2 \to \mathbb{R}^2$ is defined by $j(x_1, x_2) = (x_2, x_1)$), then the construction of the self $S$-symmetric h-sets in that system is achieved by ensuring that the homeomorphism $c_N$ satisfies the condition $j \circ c_N = c_N \circ j$.
\end{remark}

\begin{lemma}
    \label{lem:self-symmetric-covering}
    Assume that $\mathcal{N}$ is a self $S$-symmetric h-set on a Poincar\'e section $\Sigma_N$ and that $\mathcal{M}$ is an h-set (not necessarily self $S$-symmetric) on a Poincar\'e section $\Sigma_M$. Then (here we use the short-hand notation (\ref{eq:sec-covering-1}))
    \begin{align*} 
        \mathcal{N} \xRightarrow{P} \mathcal{M} & \qquad \mbox{implies} \qquad S\mathcal{N} \xRightarrow{P} \mathcal{M}, \\
        \mathcal{M} \xRightarrow{P} \mathcal{N} & \qquad \mbox{implies} \qquad \mathcal{M} \xRightarrow{P} S\mathcal{N}, \\
        \mathcal{N} \xLeftarrow{P} \mathcal{M} & \qquad \mbox{implies} \qquad S\mathcal{N} \xLeftarrow{P} \mathcal{M}, \\
        \mathcal{M} \xLeftarrow{P} \mathcal{N} & \qquad \mbox{implies} \qquad \mathcal{M} \xLeftarrow{P} S\mathcal{N}.
    \end{align*}
\end{lemma}
    
The statements of the Lemmas \ref{lem:S_backsymmetry} and \ref{lem:self-symmetric-covering} point to the following conclusion: if we are able to find a sequence of h-sets connected by the covering relations, where the first and the last of the h-sets are self $S$-symmetric, then the images of these h-sets under the symmetry $S$ are also connected by the covering relations, but in the opposite direction. Furthermore, both sequences can be connected, which leads to the setup identical to the one described in the assumptions of the Theorem \ref{thm:generic_covering_relations}.

\subsection{Ejection/collision orbit and its homoclinic}
\label{subsec:EC_orbit_and_homoclinic}
In this chapter we summarize all the key constructions from \cite{CAPINSKI2026109173} which are needed to prove the existence of the oscillatory motion to collision in the planar circular restricted three body problem for the mass ratio equal $1/82$. We begin with the discussion of the properties and the parameters of the ejection / collision orbit. Then, we present the construction of the covering relations between the h-sets $\mathcal{N}_0, ..., \mathcal{N}_{18}$ and $\mathcal{Q}_k$ for $k \in \mathbb{N}$ (see Figure \ref{fig:covering_relations_overview_detailed}). We place these sets on the Poincar\'e sections restricted to the fixed energy level in such a way that they intersect with the ejection / collision orbit or its homoclinic orbit.

Throughout this section we work with the regularized flow $\Phi_{s}^{2,h_{0}}$ generated by the Hamiltonian $\Gamma_{2,h_{0}}$, which regularizes the collision with the second mass, which in the Earth-Moon system is the Earth.

Ejection / collision orbit is determined by performing the continuation of the periodic Lyapunov orbits of the $L_1$ libration point until they reach the collision with the greater primary. Formally, the outcome of this operation is stated in the following theorem.

\begin{theorem}
    \label{thm:lyapunov_orbit_energy}
    \cite{CAPINSKI2026109173} 
    There exists positive real value $\tau$ and real value $h_{0}$ such that
    \[
        \Gamma_{2,h_{0}}(w_{0})=0,\quad\Phi_{2\tau}^{2,h_{0}}(w_{0})=w_{0}.
    \]
    where as $w_{0}=(0,0,0,\sqrt{8\mu_{2}})$ and\footnote{The exact bounds for the value of $h_0$ that we obtained and that we use in the later computations are $h_{0}\in-\big[6404647308743062,6404647308742343\big]\cdot2^{-53}$.}
    \begin{equation}
        \label{eq:h0_energy_interval}
        h_0 \in [ -0.71106, -0.71105 ].
    \end{equation}
\end{theorem}

Ejection / collision orbit is not a periodic orbit in the context of the planar circular restricted three body problem expressed in the standard coordinates ($x, y, p_x, p_y$), because in these coordinates the ejection / collision point does not belong to the phase space of the system. It is the regularization which fills the gap and makes the orbit periodic, which allows us to consider its stable set, unstable set and the homoclinic orbits.

We consider a sequence of points $w_k\in \mathbb{R}^4$ for  $k \in \{ 0, ..., K \}$ with $K = 18$. The points $w_0, .., w_3$ are placed approximately on the ejection / collision orbit and the remaining ones are placed approximately on its homoclinic orbit. The exact choice of the points $w_0,\ldots, w_K$ is written out in Table \ref{table:wk} in Appendix \ref{appendix:coordinate_systems}. At these points we will position consecutive Poincar\'e sections $\Sigma_k$ equipped with the local coordinate changes 
\[
    \psi_{k}:\mathbb{R}^{2}\to\Sigma_{k}\cap\{ \Gamma_{2,h_{0}} = 0\}, \qquad \mbox{for }k=0,\ldots,K.
\]
Local coordinate systems $\psi_1, ..., \psi_K$ have been constructed according to the description from the section \ref{subsec:poincare_sections}. The choice of the vectors $\hat u_k$ and $\hat s_k$ which specifies them is written out in Appendix \ref{appendix:coordinate_systems} (see in particular Tables \ref{table:uk} and \ref{table:sk}).

The section $\Sigma_0$ and the local coordinate change $\psi_0$ is constructed in a different way. The reason is that this section intersects with the collision manifold (see (\ref{eq:collision-circle})) and it plays a crucial role in the establishment of the orbits which begin or end at the collision.

Recall that $x_2= -\mu_1$ is the position of the second mass on the $x$-axis in the original coordinates. Let $R:\mathbb{R}\to \mathbb{R}$ be a function defined as
\begin{equation}
    \label{eq:R-def}
	R(u) := 4u^2 \left( x_{2} + u^{2} \right)^2 + 8\mu_{2} + 8 h_0^{2} + 8 \mu_{1} \frac{ u^{2} }{ |u^{2}-1| }. 
\end{equation}
We consider two fixed non-zero real values $d_1,d_2\in \mathbb{R}$ such that
\[
    d_1= 0.0598649594810129 \cdot 10^{-10} \quad \mbox{and}\quad
    d_2 = 0.997908614890024 \cdot 10^{-10}
\]
and we define
\[
	\psi_0 : \mathbb{R}^2 \to \Sigma_0 \cap \{ \Gamma_{2,h_0}=0 \}
\]
with $\Sigma_0 = \{v=0\}$ by
\[
    \psi_0 (z_1, z_2) := (u, 0, p_u, p_v(u,p_u)),
\]
where  
 \begin{align}
    u &= d_1  (z_1 + z_2), \notag \\
    p_u &= d_2  (z_1 - z_2), \label{eq:psi_0-def}\\
    p_v(u,p_u) &= 2u \left( x_2 + u^2 \right)^2 + \sqrt{ R(u) - p_u^2 }.\notag
\end{align}

Such construction asserts that the collision manifold is the image of the diagonal $z_1 + z_2 = 0$, which means that for the suitably placed h-set (e.g. the h-set $\mathcal{N}_0$ defined below), the collision manifold is both a horizontal disc and a vertical disc in that h-set (see Definitions \ref{def:horizontal-disc} and \ref{def:vertical-disc}).

For $k=0,\ldots,K$ we define a sequence of h-sets $\mathcal{N}_k$ on the sections $\{ \Sigma_k = 0 \} \cap \{\Gamma_{2,h_0}=0\}$, which are parameterized by $\psi_k$, namely:
\begin{equation}
    \label{eq:N_k}
    \mathcal{N}_k = (N_c, \psi_k),    
\end{equation}
where $N_c = [-1,1]^2$. Our choices of the vectors $\hat{u}_k, \hat{s}_k$ defining the parametrizations are appropriately rescaled so that we can choose all of these h-sets to be of the same size in the local coordinates at $w_k$.

\begin{remark}
    The h-sets $\mathcal{N}_0$, $\mathcal{N}_2$ and $\mathcal{N}_K$ are self $S$-symmetric. We achieve this, by ensuring that the origin $w \in \mathbb{R}^4$ and the direction matrix $A \in\mathbb{R}^{4 \times 4}$ for each of the local coordinate systems of these h-sets satisfies:
    \begin{equation}
        \label{eq:S_symmetry_A_matrix_conditions}
        Sw = w,\quad \hat{s} = S\hat{u},\quad \hat{n} = -S\hat{n},\quad \hat{h} = S\hat{h},
    \end{equation}
    where $A = \left[ \hat{u} \quad \hat{s} \quad \hat{n} \quad \hat{h} \right]$ (see (\ref{eq:A-form})). It is crucial to notice that the equations from (\ref{eq:S_symmetry_A_matrix_conditions}) can be satisified with the numerical computations, because the value $0$ are the operation of multiplication by $-1$ are represented exactly without any numerical errors.
\end{remark}

For the results presented below consider a sequence of Poincar\'e maps 
\begin{equation} \label{eq:PS-maps}
P^S_k:S\Sigma_k \to S\Sigma_{k-1}  \qquad \mbox{for } k=1,\ldots,K.
\end{equation}

\begin{theorem}
    \label{thm:covering_relations_1_3}
    \cite{CAPINSKI2026109173} 
    We have the following coverings (here we use the short-hand notation (\ref{eq:sec-covering-1}--\ref{eq:sec-covering-2})):
    \begin{align}
        \label{eq:covering_sequence_c1}
        &\mathcal{N}_0 \xRightarrow{P_1}
        \mathcal{N}_1 \xRightarrow{P_2}
        \mathcal{N}_2, \\
        \label{eq:covering_sequence_c2}
        &\mathcal{N}_2 \xLeftarrow{P_3}
        \mathcal{N}_3 \xLeftarrow{P_0}
        \mathcal{N}_0,
    \end{align}
    \begin{align}
        \label{eq:covering_sequence_o1}
        &\mathcal{N}_0 \xRightarrow{P_1}
        \mathcal{N}_1 \xRightarrow{P_2} 
        \mathcal{N}_2 \xLeftarrow{P_3} 
        \mathcal{N}_3 \xRightarrow{P_4} ... \xRightarrow{P_K}
        \mathcal{N}_K, \\
        \label{eq:covering_sequence_o2}
        &\mathcal{N}_K \xLeftarrow{P^S_K}
        S \mathcal{N}_{K-1} \xLeftarrow{P^S_{K-1}} ... \xLeftarrow{P^S_4}
        S \mathcal{N}_3 \xRightarrow{P^S_3}
        S \mathcal{N}_2 \xLeftarrow{P^S_2}
        S \mathcal{N}_1 \xLeftarrow{P^S_1}
        \mathcal{N}_0.
     \end{align}
\end{theorem}

\begin{theorem}
    \label{thm:covering_relations_R_Q}
    \cite{CAPINSKI2026109173} 
    There exists a sequence of the self $S$-symmetric h-sets $\mathcal{Q}_4, \mathcal{Q}_8, ...$ such that for the nonnegative integer $k$ going to infinity, the sets 
    $\mathcal{Q}_{4(k+1)}$ approach the collision arbitrarily close, but they never intersect it. Furthermore, these exists a sequence of the h-sets $\mathcal{R}_1, \mathcal{R}_2, ...$ such that the following covering relations exist:
    \begin{equation}
        \label{eq:R-Q-coverings}
        \begin{array}[c]{ccccccccc}
        \mathcal{R}_{4k+1}  & \overset{P_{2}}{\implies}& \mathcal{R}_{4k+2} & \overset{P_{3}}{\implies} & \mathcal{R}_{4k+3} & \overset{P_{0}}{\implies} & \mathcal{R}_{4\left(k+1\right)} & \overset{P_{1}}{\implies} & \mathcal{R}_{4\left(k+1\right)+1}\\
        & & &  & \mathcal{R}_{4k+3} & \overset{P_{0}}{\implies} & \mathcal{Q}_{4\left(k+1\right)} & &\\
        \end{array}
    \end{equation}
    and
    \begin{equation}
        \label{eq:gluing}
        \begin{array}[c]{rcc}
            S\mathcal{N}_{4} & \overset{P^S_{4}}{\Longrightarrow} & \mathcal{R}_{1}.\\
        \end{array}
    \end{equation}
\end{theorem}

\section{Main result}
\label{sec:main_result}
\subsection{Covering relations along the heteroclinic orbit from collision to the infinity}
\label{subsec:covrel_to_infinity}

In this section we present the construction of the sequence of the covering relations which begins with the h-set $\mathcal{N}_2$ (defined in (\ref{eq:N_k}) in chapter \ref{subsec:EC_orbit_and_homoclinic}) and which ends with the h-set $\mathcal{M}_0$ (see Figure \ref{fig:covering_relations_overview_detailed}), which intersects with the parabolic orbit and which is located at a significant distance from the center of the mass of the system. The piece of the trajectory which shadows these covering relations corresponds to the evolution of the position of the test particle from the vicinity of the ejection/collision orbit to the entring a near-parabolic orbit and a significant movement away from the center of the mass of the system.

The construction described here is analogous with respect to the construction described in the chapter \ref{subsec:EC_orbit_and_homoclinic}. Let $w_{19}, ..., w_{29}$ be the points of coordinates listed in the table \ref{table:wk} in the Appendix \ref{appendix:coordinate_systems}. We place Poincar\'e sections $\Sigma_k$ which contain the respective points. The sections are equipped with the local coordinate systems:
\[
    \psi_k : \mathbb{R}^2 \to \{ \Sigma_k = 0 \} \cap \{ \Gamma_{2,h_0} = 0 \} \quad\mathrm{for}\quad k = 19,...,29.
\]
We recall that each of the coordinate systems $\psi_k$ depends on the origin $w_k$ and the direction matrix $A_k$ (see chapter \ref{subsec:poincare_sections} for the detailed description).

The matrices $A_{19}... A_{29}$ are chosen in the following way:
\begin{equation}
    \label{eq:A_19_29}
    \begin{split}
        &A_{19} = A_3, \quad A_{20} = A_0, \quad A_{21} = A_1, \quad A_{22} = A_2, \\
        &A_{23} = A_3, \quad A_{24} = A_0, \quad A_{25} = A_1, \quad A_{26} = A_2, \\
        &A_{27} = A_3 \cdot \textrm{diag}(2.0, 1, 1, 1), \\
        &A_{28} = A_0 \cdot \textrm{diag}(4.0, 1, 1, 1), \\
        &A_{29} = A_1 \cdot \textrm{diag}(8.1, 1, 1, 1), \\
    \end{split}
\end{equation}
where $A_0, A_1, A_2, A_3$ are the matrices $\mathbb{R}^{4 \times 4}$ from the respective local coordinate systems $\psi_0, \psi_1, \psi_2, \psi_3$. The reason behind this choice is the fact that the points $w_{19}, ..., w_{29}$ are placed in the vicinity of the respective points $w_0, ..., w_3$ and the reuse of their direction matrices $A$ is sufficiently accurate for our needs. In case of the matrices $A_{27}, A_{28}, A_{29}$ we introduce an extra scaling of the first column (which corresponds to an unstable vector), which is necessary for the h-sets to be of the suitable size.

Let $k = 19, ..., 29$. We define the sequence of h-sets $\mathcal{N}_k$ on sections $\{ \Sigma_k = 0 \} \cap \{ \Gamma_{2,h_0} = 0 \}$, which are parameterized with $\psi_k$, that is
\[
    \mathcal{N}_k = (N_c, \psi_k),
\]
where $N_c = [-1,1]^2$ (see Definition \ref{def:h_set_on_section}).

\begin{theorem}
    \label{thm:covrel_periodic_to_infinity}
    We have the following coverings (here we use the short-hand notation (\ref{eq:sec-covering-1}--\ref{eq:sec-covering-2})):
    \begin{equation}
        \label{eq:covrel_periodic_to_infinity}
        \mathcal{N}_2 \xRightarrow{P_{19}} \mathcal{N}_{19} \xRightarrow{P_{20}} \mathcal{N}_{20} \xRightarrow{P_{21}} ... \xRightarrow{P_{29}} \mathcal{N}_{29},
    \end{equation}
    (see Theorem \ref{thm:covering_relations_1_3} for the definition of the h-set $\mathcal{N}_2$).
    \begin{proof}
        Existence of the covering relations (\ref{eq:covrel_periodic_to_infinity}) was proven by direct rigorous numeric validation performed with the CAPD package. Our local coordinates have been chosen in a way that the Poincar\'e maps expressed in the local coordinates are well aligned with the dynamics (see section \ref{subsec:poincare_sections}). The CAPD package implements the rigorous computation of Poincar\'e maps $P_{k}$ with the use of interval arithmetic. As discussed in section \ref{subsec:poincare_sections}, the coordinate changes $\psi_{k}$ and $\psi_{k}^{-1}$ are also computable in interval arithmetic. As a result we can compute the interval enclosures of the Poincar\'e maps in the local coordinates and validate the covering relation conditions.

        The computer assisted validation of the conditions takes less than 1 minute execution time. The code is available at \cite{pcr3bp_code}. 
    \end{proof}
\end{theorem}

Due to the fact that the covering relations in the Theorem \ref{thm:covrel_periodic_to_infinity} have been obtained by the evolution of the trajectories in the Levi-Civita coordinate system, there is a need to perform the verification to assert that the orbits which shadow these covering relations do not intersect with the collision manifold.

We define the function $\mathcal{C} : \mathbb{R}^4 \to \mathbb{R}^3$ by
\begin{equation}
    \label{eq:collision_set_condition}
    \mathcal{C} (u,v,p_u,p_v) := \left( u, v, p_u^2 + p_v^2 - 8\mu_2 \right).
\end{equation}
By (\ref{eq:collision-circle}) we see that the points $w \in \mathbb{R}^4$ satisfying $\mathcal{C}(w) = 0$ correspond to a collision with the regularized mass, which is the Earth in our case. 

\begin{theorem}
    Any orbit which shadows the sequence of covering relations (\ref{eq:covrel_periodic_to_infinity}) does not intersect with the collision manifold.
    \begin{proof}
        The proof follows from the direct numerical computations. The CAPD library provides the {\tt SolutionCurve} functionality, which allows us to obtain the rigorous bounds for the orbit enclosures.

        We therefore validate that $\mathcal{C} \ne 0$ for the entire bound of the orbits starting at $\psi_2(N_c)$, $\psi_{19} (N_c)$, $\psi_{20} (N_c), ...$ and ending at $\psi_{19} (N_c)$, $\psi_{20} (N_c)$, $\psi_{21} (N_c),...$ respectively.
        
        These computations are performed together with the numerical computation needed for the proof of Theorem \ref{eq:covrel_periodic_to_infinity}.
    \end{proof}
\end{theorem}

Let $L \in [ 0,1 ) $ and consider $\eta_L : \mathbb{R}^{2} \to \mathbb{R}^2$ defined as 
\begin{equation}
    \label{eq:eta-L}
    \eta_L(z)=\frac{1}{1+L}\left( 
    \begin{array}{cc}
        1 & -L \\ 
        -L & 1
        \end{array}
    \right)(z). 
\end{equation}
\begin{remark}
    The reason behind introducing the factor $1/(1+L)$ into the definition of $\eta_L$ is to ensure that $\eta_L (N_c) \subseteq N_c$ and $\eta^{-1}_L = \eta_{-L}$.
\end{remark}

We define the section $\Sigma_\theta$ as:
\begin{equation}
    \label{eq:sigma_theta}
    \Sigma_\theta =  \{ (\xi, \zeta, \theta, p_\theta) : \theta = -64000\pi \}
\end{equation}
and the h-set contained in that section (according to the definition \ref{def:h_set_on_section}):
\begin{equation}
    \label{eq:def_M_0}
    \mathcal{M}_0 = \left( \eta_{L} M_0, \psi_\theta \right), \quad
    M_0 = \left[ 0, 9 \cdot 2^{-20} \right] \times [75 \cdot 2^{-11}, 79 \cdot 2^{-11} ].
\end{equation}
where the mapping $\psi_\theta : N_c \to \mathbb{R}^4$ is defined as:
\begin{equation}
    \label{eq:psi_alpha_def}
    \psi_\theta (\xi, \zeta) = \left(\xi, \zeta, \theta, p_r \right), \quad p_r = E_H \left( h_0, \frac{1}{2} (\zeta + \xi), \theta, \frac{1}{2} (\zeta - \xi) \right),
\end{equation}
for $\theta = -64000\pi$ (see (\ref{eq:E_H_solution}) for the definition of the function $E_H$). The value $L$ was chosen as $21 \cdot 2^{-21}$, which is approximately $1/100000$.

The presence of the factor $\eta_{L}$ in our construction ensures that the placement of the h-set $\mathcal{M}_0$ takes into account the fact that the stable and unstable manifolds do not coindice with the directions $\xi$ and $\zeta$ exactly. As we shall see in the next chapter, h-sets defined with that factor shall correctly ``handle" the numerical errors originating from the computations.

\begin{remark}
    \label{rem:psi_theta_S_symmetry}
    The local coordiate system $\psi_\theta$ satisfies the relationship $S' \circ \psi_\theta \circ j = \psi_\theta$ (see (\ref{eq:Sp_symmetry}) for the definition of $S'$). It is implied by the symmetry of the function $E_H$ (see Remark \ref{rem:E_H_even}) and the fact that the coordinate $\theta = 64000\pi$ is equivalent $\theta = -64000\pi$ due to the $2\pi$-periodicity of this coordinate. This means that the self $S$-symmetric h-sets can be constructed in the intersection of $\Sigma_\theta$ and the constant energy level $\{ \Gamma_{2,h_0} = 0 \}$ according to the procedure described in the Remark \ref{rem:S_symmetry_cN}.
\end{remark}

\begin{theorem}
    \label{thm:covrel_jump_to_infinity}
    We have the following covering
    \begin{equation}
        \label{eq:covrel_jump_to_infinity}
        \mathcal{N}_{29} \xRightarrow{ P_{\Sigma_\theta} \circ\, \hat{\gamma} \,\circ\, \Phi_{t=1/4} \,\circ\, \gamma_2 } \mathcal{M}_0,
    \end{equation}
    where $\gamma_2$ is the Levi-Civita coordinate change defined in (\ref{eq:lc_coord_change}), $\Phi_{t=1/4}$ is the PCR3BP flow in standard coordinates for the fixed time equal $1/4$, $\hat{\gamma}$ is the coordinate change to the McGehee coordinates (see (\ref{eq:hat_gamma})) and $P_{\Sigma_\theta}$ is the Poincar\'e map onto section $\Sigma_\theta$ along the PCR3BP flow in the McGehee coordinartes.
    \begin{proof}
        Existence of the covering relation has been established with the computer assisted proof in a way analogous to the one for the Theorem \ref{thm:covrel_periodic_to_infinity}.

        The computer assisted validation of the condition takes less than 5 minutes execution time. The code is available at \cite{pcr3bp_code}. 
    \end{proof}
\end{theorem}

\begin{remark}
    In case of the orbits which shadow the covering relation (\ref{eq:covrel_jump_to_infinity}), there is no need to perform a separate proof to check that these orbits do not intersect with the collision manifold. This comes from the fact that when we perform the numerical integration to compute the values of the functions $\Phi_{t=1/4}$ and $P_{\Sigma_\theta}$ we do not use the regularized coordinates. An intersection with the collision manifold cannot take place, because otherwise, numerical computations performed with the CAPD package would not be successful. They would be terminated due to encountering the singularity.
\end{remark}

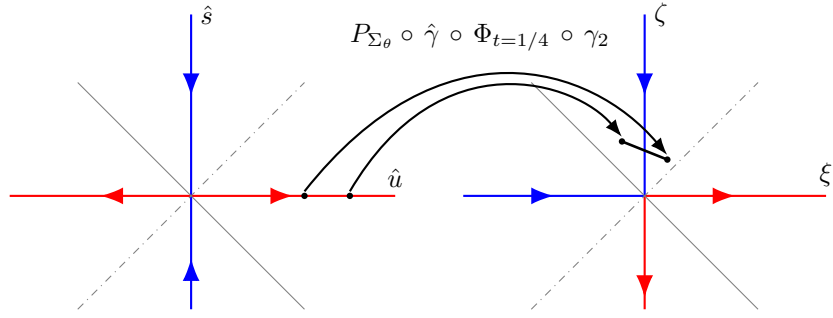
\begin{figure}
    \begin{center}
        \begin{tikzpicture}[scale=0.6]

            \begin{scope}[thick,decoration={
                markings,
                mark=at position 0.55 with {\arrowreversed{Latex[length=3mm]}}}
                ] 
                \draw[blue, postaction={decorate}] (-10,0) -- (-10,4);
            \end{scope}

            \begin{scope}[thick,decoration={
                markings,
                mark=at position 0.5 with {\arrow{Latex[length=3mm]}}}
                ] 
                \draw[red, postaction={decorate}] (-10,0) -- (-5.5,0);
            \end{scope}

            \begin{scope}[thick,decoration={
                markings,
                mark=at position 0.55 with {\arrowreversed{Latex[length=3mm]}}}
                ] 
                \draw[blue, postaction={decorate}] (-10,0) -- (-10,-2.5);
            \end{scope}

            \begin{scope}[thick,decoration={
                markings,
                mark=at position 0.5 with {\arrow{Latex[length=3mm]}}}
                ] 
                \draw[red, postaction={decorate}] (-10,0) -- (-14.0,0);
            \end{scope}

            \begin{scope}[thick,decoration={
                markings,
                mark=at position 0.5 with {\arrow{Latex[length=3mm]}}}
                ] 
                \draw[red, postaction={decorate}] (0,0) -- (4.0,0);
            \end{scope}

            \begin{scope}[thick,decoration={
                markings,
                mark=at position 0.55 with {\arrowreversed{Latex[length=3mm]}}}
                ] 
                \draw[blue, postaction={decorate}] (0,0) -- (0,4);
            \end{scope}

            \begin{scope}[thick,decoration={
                markings,
                mark=at position 0.1 with {\arrowreversed{Latex[length=3mm]}}}
                ] 
                \draw[red, postaction={decorate}] (-0.0,-2.5) -- (0,0);
            \end{scope}

            \begin{scope}[thick,decoration={
                markings,
                mark=at position 0.5 with {\arrow{Latex[length=3mm]}}}
                ] 
                \draw[blue, postaction={decorate}] (-4.0,-0.0) -- (0,0);
            \end{scope}

            \draw[thin, gray] (2.5,-2.5) -- (-2.5,2.5);

            \draw[thin, gray] (-7.5,-2.5) -- (-12.5,2.5);

            \draw[-Latex, thick] (-6.5,0.1) .. controls (-4.5,3.5) and (-1.5, 2.5) .. (-0.5,1.3);
            \draw[-Latex, thick] (-7.5,0.1) .. controls (-4.5,3.9) and (-1.2, 2.9) .. (0.5,0.9);

            \fill (-7.5,0) circle (2pt);
            \fill (-6.5,0) circle (2pt);

            \fill (-0.5,1.2) circle (2pt);
            \fill ( 0.5,0.8) circle (2pt);

            \draw[line width=1pt] (-0.5, 1.2) -- (0.5, 0.8);
            
            
            
            \draw[thin, gray, dash dot] (-2.5,-2.5) -- (2.5,2.5);

            \draw[thin, gray, dash dot] (-12.5,-2.5) -- (-7.5,2.5);

            \node[anchor=south] at (4,0) {$\xi$};
            \node[anchor=west] at (0,4) {$\zeta$};
            \node[anchor=south] at (-5.5,0) {$\hat{u}$};
            \node[anchor=west] at (-10,4) {$\hat{s}$};

            \node[anchor=west] at (-6.7,3.5) {$P_{\Sigma_\theta} \circ\, \hat{\gamma} \,\circ\, \Phi_{t=1/4} \,\circ\, \gamma_2$};

        \end{tikzpicture}
    \end{center}
    \caption{An illustration of the covering relation present in the Theorem \ref{thm:covrel_jump_to_infinity}: unstable manifold of the ejection-collision orbit intersects transversally with the stable manifold of the limit set of the parabolic orbits.}
\end{figure}

\subsection{Covering relations in the neighborhood of the parabolic fixed point}
\label{subsec:covrel_at_infinity}
In this section we present the construction of the sequence of h-sets $(\mathcal{M}_k)_{k \in \mathbb{N}}$, which intersect with the parabolic orbit and which are located arbitrarily far from the center of mass of the system for $k$ going to infinity. We also show that for each $k$ there exists an $S$-symmetric h-set $\mathcal{W}_k$, which is also located arbitrarily far from the center of the mass of the system for $k$ going to infinity and which is covered by the h-set $\mathcal{M}_{k-1}$.

The idea for the construction presented in this section is analogous to the construction described in section 5.1 in \cite{CAPINSKI2026109173}. However, the characterization of the fixed point in the neighborhood of which we constructed the sets $\mathcal{R}$ and $\mathcal{Q}$ is significantly different from the characterization of the fixed point in the neighborhood of which we shall construct the sets $\mathcal{M}$ and $\mathcal{W}$.

In case of the collision point, the hyperbolic dynamic was visible already in the derivative of the Poincar\'e map: rigorous numerical computations of the Jacobi matrix revealed that the dynamic along the first coordinate is unstable and the dynamic along the second coordinate is stable. In case of the parabolic limit point $(\xi, \zeta) = (0,0)$, the derivative of the flow with respect to the initial position is equal identity.

In order to solve this problem, we formulate the auxiliary theorem, which allows us to ``extract" the information about the dynamic of the integrated differential equation in the neighborhood of the invariant manifold, by introducing suitable additional variables. Then, we apply this theorem to the mapping between the h-sets $\mathcal{M}$ and $\mathcal{W}$, which are the subsets of section $\Sigma_\theta$ (see (\ref{eq:sigma_theta})), which asserts the existence of the suitable covering relations.

\begin{theorem}
    \label{thm:exp_trick}
    Let $n$ and $m$ be the positive integer values. We consider an ordinary differential equation:
    \begin{equation}
        \label{eq:ode_initial}
        \begin{cases}
            \dot{X} &= \mathcal{B}(X) \mathcal{A}(X,Y) X, \\
            \dot{Y} &= \mathcal{F}(X,Y), \\
        \end{cases}
    \end{equation}
    with the initial condition $X(0) = X_0 \in \mathbb{R}^n$, $Y(0) = Y_0 \in \mathbb{R}^m$, where:
    \[
        \mathcal{B} : \mathbb{R}^n \to \mathbb{R}, \quad
        \mathcal{A} : \mathbb{R}^{n+m} \to \mathbb{R}^{n \times n}, \quad
        \mathcal{F} : \mathbb{R}^{n+m} \to \mathbb{R}^m.
    \]
    The solution $X(t)$ of the equation (\ref{eq:ode_initial}) satisfies the equation:
    \begin{equation}
        \label{eq:solution_x}
        X(t) - X(0) = \mathcal{B}(X_0) U(t),
    \end{equation}
    where the functions $C(t)$ and $U(t)$ are the solution of the system of the ordinary differential equations:
    \begin{align}
        \dot{C} &= \big( D \mathcal{B}(X) \mathcal{A}(X,Y) X \big) C, \label{eq:ode_c} \\
        \dot{U} &= C \cdot \mathcal{A}(X,Y) \big( X_0 + \mathcal{B}(X_0) U \big), \label{eq:ode_gamma}
    \end{align}
    with the initial condition $C(0) = 1$ and $U(0) = 0 \in \mathbb{R}^n$.
\end{theorem}
\begin{proof}
    We notice that the function $t \mapsto \mathcal{B}(X(t))$ is the solution of the differential equation:
    \begin{equation}
        \label{eq:ode_b}
        \frac{d}{dt} \mathcal{B}(X(t)) = \big( D \mathcal{B}( X ) \mathcal{A}(X,Y) X \big) \mathcal{B}(X(t)),
    \end{equation}
    which means that $\mathcal{B}(X(t))$ can be expressed as
    \begin{equation}
        \label{eq:b_solution}
        \mathcal{B}(X(t)) = \mathcal{B}(X_0) C.
    \end{equation}

    We differentiate the equation (\ref{eq:solution_x}) with respect to $t$ and we substitute (\ref{eq:ode_gamma}) into the result of the differentiation, which leads us to
    \begin{equation}
        \label{eq:ode_x_2}
        \dot{X}(t) = \mathcal{B}(X_0) C \cdot \mathcal{A}(X,Y) \big( X_0 + \mathcal{B}(X_0) U(t) \big).
    \end{equation}

    We notice that the equation above is equivalent to the original one (due to the relationships (\ref{eq:b_solution}) and (\ref{eq:solution_x})), which ends our proof.
\end{proof}

\begin{corollary}
    \label{col:exp_trick_Poincare}
    Let $P_\Sigma (X,Y)$ be the Poincar\'e map onto section
    \[
        \Sigma (X,Y) = 0,
    \]
    along the flow $\phi_t (X,Y)$, which is induced by the system of differential equations (\ref{eq:ode_initial}), i.e.:
    \[
        \phi_t (X_0, Y_0) = \left( X(t), Y(t) \right).
    \]
    Let $\hat{P}_{\hat{\Sigma}} (X,Y,U,C)$ be the Poincar\'e map onto section
    \[
        \hat{\Sigma} (X,Y,U,C) = \Sigma(X,Y) = 0,
    \]
    along the flow $\hat{\phi}_t (X,Y,U,C)$, which is induced by the system of differential equations (\ref{eq:ode_initial}, \ref{eq:ode_c}, \ref{eq:ode_gamma}), i.e.:
    \[
        \hat{\phi}_t (X_0, Y_0, C_0, U_0) = \left( X(t), Y(t), C(t), U(t) \right).
    \]
    It is true that:
    \begin{equation}
        \label{eq:poincare_from_poincare}
        \pi_X P_\Sigma (X, Y) = X + \mathcal{B}(X) \pi_U \hat{P}_{\hat{\Sigma}} (X, Y, 1, 0),
    \end{equation}
    where $\pi_X$ and $\pi_U$ denote the projections onto the coordinates $X$ and $U$ respectively.
\end{corollary}

Let $f_\theta$ be the covering map for the h-sets contained in $\Sigma_\theta \cap \{ \Gamma_{2,h_0} = 0\}$, i.e.:
\begin{equation}
    f_\theta = \eta_{-L} \circ \psi_\theta^{-1} \circ P_{\Sigma_\theta} \circ \psi_\theta \circ \eta_{L},
\end{equation}
(see (\ref{eq:psi_alpha_def}) for the definition of $\psi_\theta$).

The inverse of $\psi_\theta$ is simply the projection onto the first two coordinates, as it was the case for the other mappings $\psi$ which parameterize other sections in our construction.

Presence of the factor $\eta_{L}$ is due to the fact that all sets contained in the section $\Sigma_\theta \cap \{ \Gamma_{2,h_0} = 0\}$ shall be scaled by it (see the definition of the set $\mathcal{M}_0$ in (\ref{eq:def_M_0})).

In the lemma below we present the outcome of the application of the corollary \ref{col:exp_trick_Poincare} to the mapping $P_{\Sigma_\theta}$, which enables us to obtain the information about the dynamic of the mapping $f_\theta$ in the neighborhood of the point $(0,0)$, which corresponds to the parabolic limit point.

\begin{lemma}
    \label{lem:f_theta_in_D}
    For $\xi, \zeta \in \mathcal{D} = \left[ -41 \cdot 2^{-11}, 19 \cdot 2^{-21} \right] \times [0, 80 \cdot 2^{-11}]$, the mapping $f_\theta$ satisfies the relationship:
    \[
        f_\theta (\xi, \zeta) \in 
        \left(Id + (\xi + \zeta)^3 \begin{pmatrix} a_\xi & -b_\xi \\ b_\zeta & -a_\zeta \end{pmatrix} \right) \begin{pmatrix} \xi \\ \zeta \end{pmatrix},
    \]
    where $a_\xi, b_\xi, a_\zeta, b_\zeta$ are the functions of $\xi, \zeta$ and their images satisfy the inclusions:
    \[
        a_\xi \left( \mathcal{D} \right), a_\zeta \left( \mathcal{D} \right) \subset \hat{a} = \left[ \frac{1608}{8192}, \frac{1609}{8192} \right],
    \]
    \[
        b_\xi \left( \mathcal{D} \right), b_\zeta \left( \mathcal{D} \right) \subset \hat{b} = \left[ \frac{3}{2^{20}}, \frac{43}{2^{20}} \right].
    \]
    \begin{proof}
        Let $(\xi, \zeta, \theta, p_\theta, \varsigma, \xi', \zeta')$ be the solutions of the extended system of the ordinary equations (\ref{eq:vector_field_aligned}) according to the Theorem \ref{thm:exp_trick}, i.e.:
        \[
            X = (\xi, \zeta), \quad Y = (\theta, p_\theta), \quad C = \varsigma, \quad U = (\xi', \zeta').
        \]
        Similarily, we extend the mapping $\psi_\theta$, by the introduction of the mapping $\hat{\psi}_\theta$:
        \[
            \hat{\psi}_\theta (\xi, \zeta) = \left(\xi, \zeta, \theta, E_H \left( h_0, \frac{1}{2} (\zeta + \xi), \theta, \frac{1}{2} (\zeta - \xi) \right), 1, 0, 0 \right), \quad \theta = 0
        \]
        and we make use of the relationship (\ref{eq:poincare_from_poincare}) to notice that:
        \[
            f_\theta (\xi, \zeta) =
                (\xi, \zeta) + (\xi + \zeta)^3 \left( \pi_{\xi', \zeta'} \circ \hat{P}_{\hat{\Sigma}_\theta} \circ \hat{\psi}_\theta \right) (\xi, \zeta).
        \]
        The fact that $\left( \pi_{\xi', \zeta'} \circ \hat{P}_{\hat{\Sigma}_\theta} \circ \hat{\psi}_\theta \right)(0,0) = (0,0)$ implies that:
        \[
            \left( \pi_{\xi', \zeta'} \circ \hat{P}_{\hat{\Sigma}_\theta} \circ \hat{\psi}_\theta \right) (\xi, \zeta) \in 
            \underbrace{D_{\xi', \zeta'} \left( \pi_{\xi', \zeta'} \circ  \hat{P}_{\hat{\Sigma}_\theta} \circ \hat{\psi}_\theta \right) (\xi, \zeta)}_{\subset \begin{pmatrix} \hat{a} & -\hat{b} \\ \hat{b} & -\hat{a} \end{pmatrix} } \begin{pmatrix} \xi \\ \zeta \end{pmatrix}.
        \]
        Numeric values of the intervals $\hat{a}$ oraz $\hat{b}$ are determined with the use of rigorous numeric computations provided by the CAPD library, which concludes our proof.

        The computer assisted validation of the conditions takes less than 1 minute execution time. The code is available at \cite{pcr3bp_code}. 
    \end{proof}
\end{lemma}

We can perform an additional nonrigorous verification of the correctness of the lemma \ref{lem:f_theta_in_D}. The differential equation (\ref{eq:ode_c}) implies that the value $C = \varsigma$ is approximately constant and equal to $1$. This is the consequence of the presence of the term $D\mathcal{B}(X) = (3\xi^2, 3\zeta^2)$, which approaches zero quickly in the neighborhood of the point $(\xi,\zeta) = (0,0)$. This means that the expected value of the coordinate $U = (\xi', \zeta')$ is approximately equal $1/32 \cdot \tau \cdot (\xi,\zeta)$, where $\tau$ is the integration time. Presence of the factor $1/32$ comes from the form of the differential equations for variables $\xi$ and $\zeta$ (see (\ref{eq:vector_field_aligned})).

Integration time for the Poincar\'e mapping from the section $\Sigma_\theta$ onto $\Sigma_\theta$ is approximately equal $2\pi$, which comes from the form of the differential equation for the variable $\theta$ (see (\ref{eq:vector_field_aligned})). As a result, we get the expected approximate value of the interval $\hat{a}$ equal $\pi / 16$, which agrees with the value obtained in the lemma \ref{lem:f_theta_in_D}.

We notice that the bounds for the off-diagonal terms (interval $\hat{b}$) are not symmetric: the value $0$ is not contained in them. This is the consequence of the use of the term $\eta_{L}$, which ``tilts" the stable and unstable directions of the point $(\xi,\zeta) = (0,0)$. We shall see that this is an essential element of our construction to get the correct image bounds which shall imply the existence of the covering relations.

Before we present our construction of the sequence of the h-sets which approaches the point $(\xi,\zeta) = (0,0)$, we present an auxiliary lemma.

\begin{lemma}
    \label{lem:u_sequence}
    Let $u_n$ be the sequence of real numbers defined with the formula:
    \[
        u_n^{-\kappa} - u_0^{-\kappa} = \kappa \nu n,
    \]
    where $\kappa > 0$, $\nu > 0$, $u_0 > 0$. Then
    \begin{equation}
        \label{eq:un_1}
        u_{n+1} \geq u_n - \nu u_n^{\kappa+1}.
    \end{equation}
    \begin{proof}
        Recursive formula for the sequence $u_n$ can be obtained straight from the equation
        \[
            u_{n+1}^{-\kappa} - u_n^{-\kappa} = \kappa \nu,
        \]
        which gives:
        \[
            u_{n+1} = u_n \left( 1 + \kappa \nu u_n^\kappa \right)^{-1/\kappa}.
        \]
        We define the function $\delta(z) = (1+\kappa \nu z)^{-1/\kappa}$ and we notice that it is convex. The inequality (\ref{eq:un_1}) is directly implied by the inequality 
        \[
            \delta(z) \geq \delta(0) + D\delta(0) z, \quad z \in \mathbb{R},
        \]
        which is the general property of the convex functions.
    \end{proof}
\end{lemma}

\begin{theorem}
    \label{thm:M_covrel}
    Let $\mathcal{M}$ be the h-set on section $\Sigma_\theta$ defined by:
    \[
        \mathcal{M} = \big( \eta_{L} [0, \xi_M] \times [\zeta_e, \zeta_M], \psi_\theta \big),
    \]
    where $[0, \xi_M] \times [\zeta_e, \zeta_M] \subset \mathcal{D}$ and:
    \begin{equation}
        \label{eq:w_b_a_xi}
        \xi_M > \frac{\hat{b}}{\hat{a}} \zeta_M,
    \end{equation}
    \begin{equation}
        \label{eq:zeta_e_M}
        0 < \zeta_e < \zeta_M,
    \end{equation}
    \begin{equation}
        \label{eq:a_w_xi_3}
        \hat{a} \left( \xi_M + \zeta_M \right)^3 \subset (0, 1).
    \end{equation}
    Let $\lambda$ be the fixed positive real value such that:
    \begin{equation}
        \label{eq:a_lambda_b}
        \hat{a}\lambda > \hat{b},
    \end{equation}
    \begin{equation}
        \label{eq:y_y3_inc}
        3 \zeta_M^2 \left( \hat{a}\lambda - \hat{b} \right) < 1,
    \end{equation}
    \begin{equation}
        \label{eq:a_lambda_b_w_xi}
        1 - \left( \hat{a} \lambda - \hat{b} \right) \xi_M \zeta_M^2 > 0,
    \end{equation}
    \begin{equation}
        \label{eq:lambda_xi_M_3}
        \lambda + \xi_M^3 \left( \lambda + 1 \right)^3 \hat{b} < 1.
    \end{equation}
    Then, there exists a positive real value $\zeta_e'$ and the finite sequence of the covering relations:
    \[
        \mathcal{M} \xRightarrow{P_\theta} ... \xRightarrow{P_\theta} \mathcal{M'},
    \]
    where $\mathcal{M'} = \left( \eta_{L} \left[ 0, \xi_M \right] \times \left[\zeta_e', \left(\lambda + \xi_M^3 (\lambda + 1)^3 \hat{b} \right) \xi_M \right], \psi_\theta \right)$ (see the Figure \ref{fig:covrel_at_infinity_single} for the depiction of this theorem).
    \begin{proof}
        The assumption (\ref{eq:a_lambda_b}) implies that there exists a positive real value $\lambda_C$, which satisfies
        \begin{equation}
            \label{eq:c_lambda}
            \lambda_C < \left( \hat{a}\lambda - \hat{b} \right) \xi_M.
        \end{equation}
        The assumptions (\ref{eq:a_w_xi_3}) and (\ref{eq:a_lambda_b_w_xi}) imply that there exists a positive real value $\lambda_E$, which satisfies
        \begin{equation}
            \label{eq:lambda_2}
            \lambda_E < 1 - \hat{a} \left( \xi_M+\zeta_M \right)^3,
        \end{equation}
        \begin{equation}
            \label{eq:lambda_2_E}
            \lambda_E < 1 - \lambda_C \xi_M \zeta_M^2.
        \end{equation}
        We define the sequence of h-sets $\mathcal{M}'_n$ by:
        \[
            \mathcal{M}'_n = \big( \eta_{L} [0, \xi_M] \times [e_n, \zeta_n], \psi_\theta \big),
        \]
        where $e_n$ and $\zeta_n$ are the sequences of the real numbers which satisfy
        \[
            \zeta_n^{-2} - \zeta_0^{-2} = 2 \lambda_C  n, \quad \zeta_0 = \zeta_M, \quad
            e_n = \lambda_E^n \zeta_e.
        \]
        According to the lemma \ref{lem:u_sequence}, the sequence $\zeta_n$ satisfies the inequality:
        \[
            \zeta_{n+1} > \zeta_n - \lambda_C \zeta_n^3.
        \]
        The sequence of h-sets $\mathcal{M}'_n$ is well defined as long as:
        \[
            \zeta_n > e_n
        \]
        is satisfies for each $n$. This is indeed the case, because $\zeta_0 > e_0$ according to the assumption (\ref{eq:zeta_e_M}), and for the next elements, the following recursive relation is true:
        \[
            \zeta_n > e_n \implies \zeta_{n+1} >
            \zeta_n - \lambda_C \zeta_n^3 >
            e_n \left( 1 - \lambda_C \zeta_M^2 \right) >
            e_n \lambda_E = e_{n+1},
        \]
        where the last inequality is implied by the inequality (\ref{eq:lambda_2_E}).

        We show that $\mathcal{M}'_n \xRightarrow{P_\theta} \mathcal{M}'_{n+1}$ for $\zeta_n \geq \left(\lambda + \xi_M^3 (\lambda + 1)^3 \hat{b} \right) \xi_M$ by the direct verification of the conditions of the covering relation. We introduce variables $\tilde{\xi} ,\tilde{\zeta}$ such that:
        \[
            \tilde{\xi} \in [0,\xi_M] \quad\text{oraz}\quad \tilde{\zeta} \in [e_n, \zeta_n].
        \]
        We simplify the notation by skipping the arguments of the functions $a_\xi$, $b_\xi$, $a_\zeta$ and $b_\zeta$. In each case they are identical with the arguments of the function $f_\theta$.

        Left expansion condition is directly implied by the estimate:
        \[
            \pi_\xi f_\theta \left( 0,\tilde{\zeta} \right) = \tilde{\zeta}^3 \cdot \left(-b_\xi \tilde{\zeta} \right) < 0.
        \]
        We verify the right expansion condition in an anlogous way:
        \[
            \pi_\xi f_\theta \left( \xi_M,\tilde{\zeta} \right) =
            \xi_M + \left( \xi_M+\tilde{\zeta} \right)^3 \left( a_\xi \xi_M-b_\xi \tilde{\zeta} \right) > \xi_M,
        \]
        where we made use of the assumption (\ref{eq:w_b_a_xi}). The contraction condition takes form:
        \begin{equation}
            \label{eq:m_covrel_contraction}
            e_{n+1} < \pi_\zeta f_\theta \big( [0,\xi_M], [e_n, \zeta_n] \big) < \zeta_{n+1}.
        \end{equation}
        We verify the left inequality in (\ref{eq:m_covrel_contraction}) by performing the following estimate:
        \begin{equation*}
            \begin{split}
                \pi_\zeta f_\theta \left( \tilde{\xi},\tilde{\zeta} \right) &= 
                \tilde{\zeta} + \left(\tilde{\xi}+\tilde{\zeta} \right)^3 \left(b_\zeta \tilde{\xi} - a_\zeta \tilde{\zeta} \right)\\ &\geq 
                \tilde{\zeta} - a_\zeta \left( \tilde{\xi}+\tilde{\zeta} \right)^3 \tilde{\zeta} \geq
                \tilde{\zeta} \left(1 - a_\zeta \left(\xi_M+\tilde{\zeta} \right)^3 \right) >
                e_n \lambda_E = e_{n+1},
            \end{split}
        \end{equation*}
        where the last inequality is directly implied by the inequality (\ref{eq:lambda_2}).
        
        In order to verify the right inequality in (\ref{eq:m_covrel_contraction}) we consider two cases. If $\tilde{\zeta} < \lambda \xi_M$, then we get:
        \begin{equation}
            \label{eq:pi_zeta_contr}
            \begin{split}
                \pi_\zeta f_\theta \left( \tilde{\xi},\tilde{\zeta} \right) &\leq
                \tilde{\zeta} + \left( \tilde{\xi}+\tilde{\zeta} \right)^3 \cdot b_\zeta \tilde{\xi}\\ &\leq
                \tilde{\zeta} + \left( \tilde{\zeta}+\xi_M \right)^3 \cdot b_\zeta \xi_M <
                \xi_M \left( \lambda + \xi_M^3 (\lambda + 1)^3 b_\zeta \right).
            \end{split}
        \end{equation}
        If $\tilde{\zeta} \geq \lambda \xi_M$:
        \begin{equation*}
            \begin{split}
                \pi_\zeta f_\theta \left( \tilde{\xi},\tilde{\zeta} \right) &\leq
                \tilde{\zeta} + \left( \tilde{\xi}+\tilde{\zeta} \right)^3 \left( b_\zeta - a_\zeta \lambda \right) \xi_M \\ &\leq
                \tilde{\zeta} - \tilde{\zeta}^3 \left( a_\zeta \lambda - b_\zeta \right)\xi_M \leq
                \zeta_n - \zeta_n^3 \left(a_\zeta \lambda - b_\zeta \right) \xi_M,
            \end{split}
        \end{equation*}
        where the last inequality is the consequence of the fact that the function
        \[
            \tilde{\zeta} \mapsto \tilde{\zeta} - \tilde{\zeta}^3 (a_\zeta \lambda - b_\zeta)\xi_M
        \]
        is strictly increasing, which is implied by the assumption (\ref{eq:y_y3_inc}). We notice that the inequality (\ref{eq:c_lambda}) implies that:
        \[
            \zeta_n - \zeta_n^3 (a_\zeta \lambda - b_\zeta )\xi_M < \zeta_n - \lambda_C \zeta_n^3 < \zeta_{n+1}.
        \]
        The sequence $\zeta_n$ converges to zero, which means that the sufficiently large $n$, we get:
        \[
            \pi_\zeta f_\theta \left( \tilde{\xi},\tilde{\zeta} \right) <
            \xi_M \left( \lambda + \xi_M^3 (\lambda + 1)^3 b_\zeta \right)
        \]
        according to (\ref{eq:pi_zeta_contr}), what together with the assumption (\ref{eq:lambda_xi_M_3}) concludes the proof.
    \end{proof}
\end{theorem}

\begin{figure}
    \begin{center}
        \begin{tikzpicture}[scale=0.6]

            \filldraw[thick, blue!15] (0,0) -- (0,7.5) -- (1.5,7.5) -- cycle;
            \filldraw[thick, blue!15] (0,0) -- (-4.0,0.0) -- (-4.0,-0.8) -- cycle;
            
            \filldraw[thick, red!15] (0,0) -- (5.0,0) -- (5.0,1.0) -- cycle;
            \filldraw[thick, red!15] (0,0) -- (0,-2.5) -- (-0.5,-2.5) -- cycle;

            \draw[thin, black, dashed] (-2.5,0) rectangle (4.5,6.5);
            \node[anchor=south west] at (4.5,6.5) {$\mathcal{D}$};

            \begin{scope}[thick,decoration={
                markings,
                mark=at position 0.5 with {\arrow{Latex[length=3mm]}}}
                ] 
                \draw[red, postaction={decorate}] (0,0) -- (5.0,0.5);
            \end{scope}

            \begin{scope}[thick,decoration={
                markings,
                mark=at position 0.55 with {\arrowreversed{Latex[length=3mm]}}}
                ] 
                \draw[blue, postaction={decorate}] (0,0) -- (0.75,7.5);
            \end{scope}

            \begin{scope}[thick,decoration={
                markings,
                mark=at position 0.1 with {\arrowreversed{Latex[length=3mm]}}}
                ] 
                \draw[red, postaction={decorate}] (-0.25,-2.5) -- (0,0);
            \end{scope}

            \begin{scope}[thick,decoration={
                markings,
                mark=at position 0.5 with {\arrow{Latex[length=3mm]}}}
                ] 
                \draw[blue, postaction={decorate}] (-4.0,-0.4) -- (0,0);
            \end{scope}

            \node[anchor=center, rotate=-45] at (2.0, -2.5){ $\chi = 0$ };
            \node[anchor=center, rotate=45] at (-2.0, -2.5){ $\xi = \zeta$ };

            \draw[thin, -] (0, 0) -- (5.0, 1.0) node[above right]{$\zeta = \lambda \xi$};
            \draw[thin, -] (0, 0) -- (1.5, 7.5) node[right]{$\xi = \lambda \zeta$};
            \draw[thin, -] (3, 5.5) -- (3, 0) node[below]{$\xi_M$};

            \draw[thin, gray, dashed] (0, 0) -- (3, 5.5);

            \draw[thin, gray, dashed] (2.5,-2.5) -- (-2.5,2.5);
            
            \draw[thick, -Latex] (-5.0, 0) -- (6.5, 0) node[right]{$\xi$};
            
            \draw[thick, -Latex] (0, -3,0) -- (0, 8.5) node[right]{$\zeta$};
            
            \draw[thin, gray, dash dot] (-2.5,-2.5) -- (5,5);

            \node[anchor=east] at (0, 5.5){ $\zeta_M$ };
            \node[anchor=east] at (0, 4.0){ $\zeta_e$ };
            
            \node[anchor=east] at (-3.0, 2.0){ $\left(\lambda + \xi_M^3 (\lambda + 1)^3 \hat{b} \right) \xi_M$ };
            \draw[thin, -Latex] (-3.0, 2.0) -- (-0.2, 2.0);
            
            \node[anchor=east] at (-3.0, 0.5){ $\zeta'_e$ };
            \draw[thin, -Latex] (-3.0, 0.5) -- (-0.2, 0.5);
            
            \node[anchor=west] at (3, 5.0){ $\mathcal{M}$ };
            \hsetrect{0}{3.0}{4.0}{5.5}{red}

            \node[anchor=west] at (3, 1.7){ $\mathcal{M}'$ };
            \hsetrect{0}{3.0}{0.4}{2.1}{green}

        \end{tikzpicture}
    \end{center}
    \caption{An illustration of the Theorem \ref{thm:M_covrel}: we construct the finite sequence of the h-sets connected with the covering relations. The sequence begins with the h-set $\mathcal{M}$ and ends with the h-set $\mathcal{M}'$.}
    \label{fig:covrel_at_infinity_single}
\end{figure}
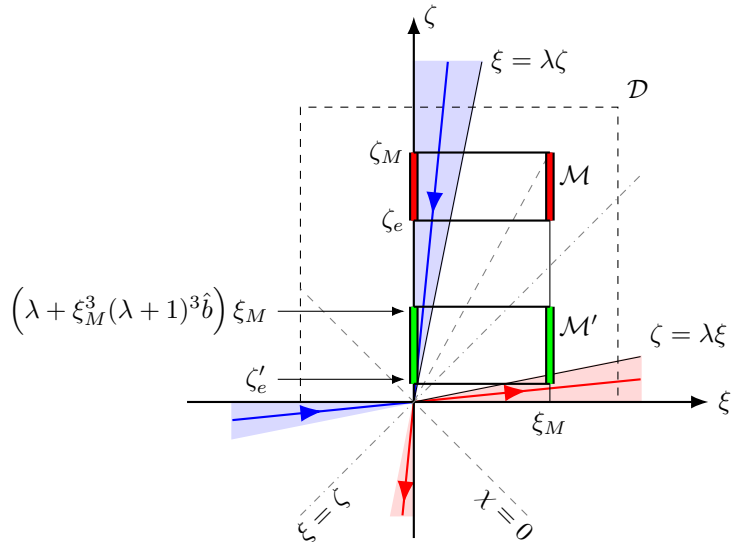

\begin{corollary}
    \label{col:covrel_M_W}
    There exists a positive real sequence $\varepsilon_n$ which converges to zero, and the sequences of h-sets $\mathcal{M}_n$ and $\mathcal{W}_n$ defined with the formulas:
    \[
        \mathcal{M}_n = \Big( \eta_{L} \left[ 0, \lambda_S^n \xi_M \right] \times \left[ \varepsilon_n, \lambda_S^n \zeta_M \right], \psi_\theta \Big),
    \]
    \[
        \mathcal{W}_n = \Big( \eta_{L} \left[ \varepsilon_n, \lambda_S^n \zeta_M \right] \times \left[ \varepsilon_n, \lambda_S^n \zeta_M \right], \psi_\theta \Big),
    \]
    where $\lambda_S = \sup \left(\lambda + \xi_M^3 (\lambda + 1)^3 \hat{b} \right) \xi_M / \zeta_M$ and $\mathcal{M}_0 = [0, \xi_M] \times [\varepsilon_0, \zeta_M]$ is the h-set defined in (\ref{eq:def_M_0}), such that there exist sequences of covering relations:
    \[
        \mathcal{M}_n \xRightarrow{P_\theta} ... \xRightarrow{P_\theta} \mathcal{M}_{n+1},
    \]
    \[
        \mathcal{M}_n \xRightarrow{P_\theta} ... \xRightarrow{P_\theta} \mathcal{W}_{n+1}
    \]
    (see the Figure \ref{fig:covrel_at_infinity_sequence} for the depiction of this theorem).
    \begin{proof}
        We take $\lambda = 7 \cdot 2^{-15}$ and we make use of the rigorous numeric computations to verify that the assumptions of the Theorem \ref{thm:M_covrel} are fulfilled for the h-set $\mathcal{M}_0$, i.e. the inequalities (\ref{eq:w_b_a_xi})--(\ref{eq:lambda_xi_M_3}) are satisfied.

        We notice that if the assumptions are satified for the h-set $\mathcal{M}_0$ then they are satisfied for each $\mathcal{M}_n$, for the chosen value of $\lambda$. 

        The existence of the sequence of the covering relations for some positive real value $\varepsilon_n'$:
        \[
            \mathcal{M}_n \xRightarrow{P_\theta} ... \xRightarrow{P_\theta} \Big( \eta_{L} \left[ 0, \lambda_S^n \xi_M \right] \times \left[ \varepsilon_n', \lambda_S^{n+1} \zeta_M \right], \psi_\theta \Big),
        \]
        for each $\mathcal{M}_n$ is directly implied by the Theorem \ref{thm:M_covrel} which concludes our proof.
    \end{proof}
\end{corollary}

\begin{remark}
    Every h-set $\mathcal{W}_n$ is self $S$-symmetric which is due to the fact that its underlying local coordinate system is $\psi_\theta$ (see Remark \ref{rem:psi_theta_S_symmetry} for details).
\end{remark}

\begin{figure}[H]
    \begin{center}
        \begin{tikzpicture}[scale=0.7]

            \filldraw[thick, blue!15] (0,0) -- (0,7.5) -- (1.5,7.5) -- cycle;
            \filldraw[thick, blue!15] (0,0) -- (-4.0,0.0) -- (-4.0,-0.8) -- cycle;
            
            \filldraw[thick, red!15] (0,0) -- (5.0,0) -- (5.0,1.0) -- cycle;
            \filldraw[thick, red!15] (0,0) -- (0,-2.5) -- (-0.5,-2.5) -- cycle;

            \draw[thin, black, dashed] (-2.5,0) rectangle (4.5,6.5);
            \node[anchor=south west] at (4.5,6.5) {$\mathcal{D}$};

            \begin{scope}[thick,decoration={
                markings,
                mark=at position 0.5 with {\arrow{Latex[length=3mm]}}}
                ] 
                \draw[red, postaction={decorate}] (0,0) -- (5.0,0.5);
            \end{scope}

            \begin{scope}[thick,decoration={
                markings,
                mark=at position 0.55 with {\arrowreversed{Latex[length=3mm]}}}
                ] 
                \draw[blue, postaction={decorate}] (0,0) -- (0.75,7.5);
            \end{scope}

            \begin{scope}[thick,decoration={
                markings,
                mark=at position 0.1 with {\arrowreversed{Latex[length=3mm]}}}
                ] 
                \draw[red, postaction={decorate}] (-0.25,-2.5) -- (0,0);
            \end{scope}

            \begin{scope}[thick,decoration={
                markings,
                mark=at position 0.5 with {\arrow{Latex[length=3mm]}}}
                ] 
                \draw[blue, postaction={decorate}] (-4.0,-0.4) -- (0,0);
            \end{scope}

            \node[anchor=center, rotate=-45] at (2.0, -2.5){ $\chi = 0$ };
            \node[anchor=center, rotate=45] at (-2.0, -2.5){ $\xi = \zeta$ };

            \draw[thin, -] (0, 0) -- (5.0, 1.0) node[above right]{$\zeta = \lambda \xi$};
            \draw[thin, -] (0, 0) -- (1.5, 7.5) node[right]{$\xi = \lambda \zeta$};
            \draw[thin, -] (3, 5.5) -- (3, 0) node[below]{$\xi_M$};
            \draw[thin, -] (6.3/5.5, 2.1) -- (6.3/5.5, 0) node[below]{$\lambda_S \xi_M$};

            \draw[thin, gray, dashed] (0, 0) -- (3, 5.5);

            \draw[thin, gray, dashed] (2.5,-2.5) -- (-2.5,2.5);
            
            \draw[thick, -Latex] (-5.0, 0) -- (6.5, 0) node[right]{$\xi$};
            
            \draw[thick, -Latex] (0, -3.0) -- (0, 8.5) node[right]{$\zeta$};
            
            \draw[thin, gray, dash dot] (-2.5,-2.5) -- (5,5);

            \node[anchor=east] at (0, 5.5){ $\zeta_M$ };
            \node[anchor=east] at (0, 4.0){ $\varepsilon_0$ };
            
            \node[anchor=east] at (-3.0, 2.2){ $\lambda_S \zeta_M$ };
            \draw[thin, -Latex] (-3.0, 2.1) -- (-0.2, 2.1);

            \node[anchor=east] at (-3.0, 0.5){ $\varepsilon_1$ };
            \draw[thin, -Latex] (-3.0, 0.5) -- (-0.2, 0.5);
            
            \node[anchor=west] at (3, 5.0){ $\mathcal{M}_0$ };
            \hsetrect{0}{3.0}{4.0}{5.5}{red}

            \node[anchor=east] at (0.0, 1.3){ $\mathcal{M}_1$ };
            \hsetrect{0.4}{2.1}{0.4}{2.1}{green}

            \node[anchor=south] at (2.1, 2.1){ $\mathcal{W}_1$ };
            \hsetrect{0}{6.3/5.5}{0.4}{2.1}{blue}

        \end{tikzpicture}
    \end{center}
    \caption{An illustration of the Corollary \ref{col:covrel_M_W} for the initial elements of the sequences $\mathcal{M}_k$ and $\mathcal{W}_k$: h-set $\mathcal{M}_0$ (red) $f_\theta$-covers h-set $\mathcal{M}_1$ (blue) and an $S$-symmetric h-set $\mathcal{W}_1$ (green). We emphasize that all sets $\mathcal{M}_k$ and $\mathcal{W}_k$ are the subsets of the domain $\mathcal{D}$.}
    \label{fig:covrel_at_infinity_sequence}
\end{figure}
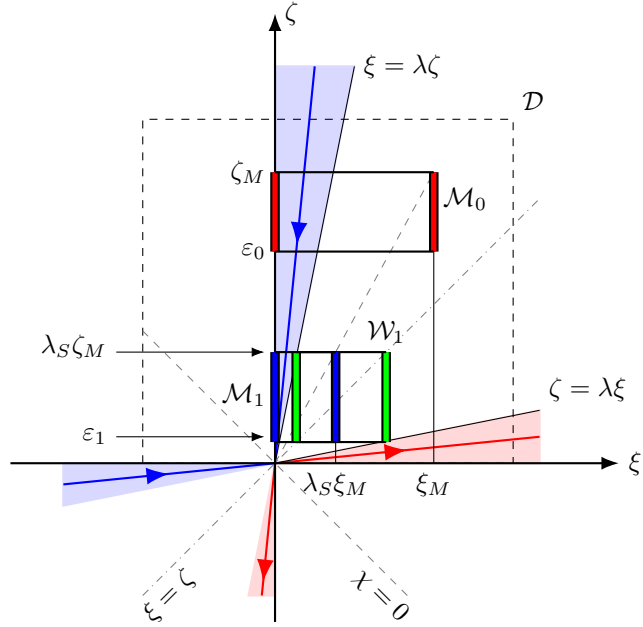

We conclude this section by discussing the subject of the hyperbolic evolution. We consider left exit set for the h-set $\mathcal{M}_0$, which is the disc $\mathcal{M}_0^l$. We shall show that the image of the set $\mathcal{M}_0^l$ under the composition of the function $f_\theta$ an arbitrary number of times shall be the subset of a suitable cone. This shall imply that the trajectories are the hyperbolic trajectories.

\begin{theorem}
    \label{lem:hyperbolic_evolution}
    For each point $\left( 0,\tilde{\zeta} \right) \in M_0^l$ there exists $\omega > 0$ such that:
    \[
        f_\theta^n \left( 0,\tilde{\zeta} \right) \xrightarrow{n \to \infty} (-\omega,\omega)
    \]
    \begin{proof}
        We notice that
        \[
            \cot \measuredangle \left( f_\theta (\xi,\zeta) - (\xi,\zeta) \right) \in \frac{\hat{a} \xi/\zeta - \hat{b}}{\hat{b} \xi/\zeta - \hat{a}},
        \]
        where $\measuredangle$ is the argument of the vector on the plane (we recall Lemma \ref{lem:hyperbolic_evolution} for the definition of $\hat{a}$ and $\hat{b}$). The inclusion $\xi/\zeta \in [-1,0]$ is true for the points $(\xi,\zeta)$ which satify
        \[
            \measuredangle(\xi,\zeta) \in \left[ \frac{\pi}{2}, \frac{3\pi}{4} \right], \quad\text{for}\quad \zeta > 0.
        \]
        We make use of this fact and we use rigorous numerical computations to get the bounds:
        \[
            \cot \measuredangle \left( f_\theta (\xi,\zeta) - (\xi,\zeta) \right) \in \left[ \frac{15}{2^{20}}, 1 +\frac{885}{2^{20}} \right],
        \]
        which means that for each point $\left( 0,\tilde{\zeta} \right) \in M_0^l$, the orbit $f_\theta^n \left( 0,\tilde{\zeta} \right)$ is the subset of the cone
        \[
            \frac{15}{2^{20}} \, \left(\zeta - \tilde{\zeta}\right) \geq \xi \geq \left( 1 +\frac{885}{2^{20}} \right) \left(\zeta - \tilde{\zeta}\right) \quad\text{for}\quad \xi \leq 0,
        \]
        which concludes our proof (see Figure \ref{fig:hyperbolic_evolution}).
    \end{proof}
\end{theorem}

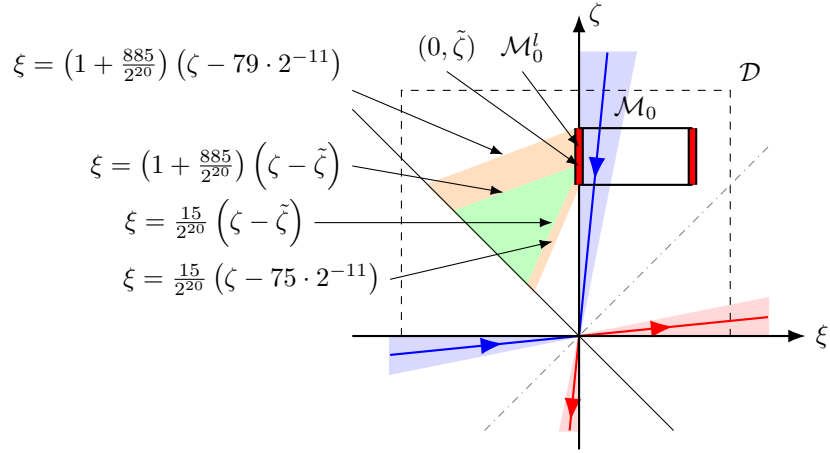
\begin{figure}[h]
    \begin{center}
        \begin{tikzpicture}[scale=0.5]

            \filldraw[thick, blue!15] (0,0) -- (0,7.5) -- (1.5,7.5) -- cycle;
            \filldraw[thick, blue!15] (0,0) -- (-5.0,0.0) -- (-5.0,-1.0) -- cycle;
            
            \filldraw[thick, orange!25] (0,4.0) -- (0,5.5) -- (-4.0,4.0) -- (-1.2,1.2) -- cycle;

            \filldraw[thick, green!25] (0,4.5) -- (-3.3,3.3) -- (-1.4,1.4) -- cycle;
            
            \filldraw[thick, red!15] (0,0) -- (5.0,0) -- (5.0,1.0) -- cycle;
            \filldraw[thick, red!15] (0,0) -- (0,-2.5) -- (-0.5,-2.5) -- cycle;

            \draw[thin, black, dashed] (-4.7,0) rectangle (4,6.5);
            \node[anchor=south west] at (4,6.5) {$\mathcal{D}$};

            \begin{scope}[thick,decoration={
                markings,
                mark=at position 0.5 with {\arrow{Latex[length=3mm]}}}
                ] 
                \draw[red, postaction={decorate}] (0,0) -- (5.0,0.5);
            \end{scope}

            \begin{scope}[thick,decoration={
                markings,
                mark=at position 0.55 with {\arrowreversed{Latex[length=3mm]}}}
                ] 
                \draw[blue, postaction={decorate}] (0,0) -- (0.75,7.5);
            \end{scope}

            \begin{scope}[thick,decoration={
                markings,
                mark=at position 0.1 with {\arrowreversed{Latex[length=3mm]}}}
                ] 
                \draw[red, postaction={decorate}] (-0.25,-2.5) -- (0,0);
            \end{scope}

            \begin{scope}[thick,decoration={
                markings,
                mark=at position 0.6 with {\arrow{Latex[length=3mm]}}}
                ] 
                \draw[blue, postaction={decorate}] (-5.0,-0.5) -- (0,0);
            \end{scope}

            \draw[thick, -Latex] (-6,0, 0) -- (6.0, 0) node[right]{$\xi$};
            
            \draw[thick, -Latex] (0, -3,0) -- (0, 8.5) node[right]{$\zeta$};
            
            \draw[thin, gray, dash dot] (-2.5,-2.5) -- (5,5);
            
            \draw[thin, black] (2.5,-2.5) -- (-6.0,6.0);
            
            \hsetrect{0}{3.0}{4.0}{5.5}{red}
            \node[anchor=south] at (1.5,5.5) {$\mathcal{M}_0$};
            
            \node[anchor=south] at (-1.5,7.0) {$\mathcal{M}^l_0$};
            \draw[thin, black, -Latex] (-1.5,7.0) -- (0,5.0);

            \node[anchor=south] at (-3.5,7.0) {$(0, \tilde{\zeta})$};
            \draw[thin, black, -Latex] (-3.5,7.0) -- (0,4.5);

            \node[anchor=south east] at (-6.0,6.5) {$\xi = \left( 1 +\frac{885}{2^{20}} \right) \left( \zeta - 79 \cdot 2^{-11} \right)$};
            \draw[thin, black, -Latex] (-6.0,6.5) -- (-2.0,4.8);

            \node[anchor=east] at (-6.0,4.5) {$\xi = \left( 1 +\frac{885}{2^{20}} \right) \left( \zeta - \tilde{\zeta} \right)$};
            \draw[thin, black, -Latex] (-6.0,4.5) -- (-2.0,3.8);

            \node[anchor=east] at (-7.0,3.0) {$\xi = \frac{15}{2^{20}} \left( \zeta - \tilde{\zeta} \right)$};
            \draw[thin, black, -Latex] (-7.0,3.0) -- (-0.7,3.0);

            \node[anchor=east] at (-5.0,1.5) {$\xi = \frac{15}{2^{20}} \left( \zeta - 75 \cdot 2^{-11} \right)$};
            \draw[thin, black, -Latex] (-5.0,1.5) -- (-0.7,2.5);
            
        \end{tikzpicture}
    \end{center}
    \caption{An illustration of the Theorem \ref{lem:hyperbolic_evolution}. The orbit which begins at the point $\left( 0, \tilde{\zeta}\right)$ is the subset of the green cone, orbits which begin in the set $\mathcal{M}_0^l$ are the subsets of the orange cone. We emphasize that both cones are the subsets of the domain $\mathcal{D}$, which asserts the correctness of the numeric computations (see the assumptions of the Lemma \ref{lem:f_theta_in_D}). The values $75 \cdot 2^{-11}$ and $79 \cdot 2^{-11}$ are the $\zeta$ coordinates of the exit set $\mathcal{M}_0^l$ (see (\ref{eq:def_M_0}))}
    \label{fig:hyperbolic_evolution}
\end{figure}

\subsection{The proof of Theorem \ref{thm:primary_result}}
\label{subsec:main_proof}
We are ready to present the proof of the main result contained in this work.

We introduce the notation $\mathcal{N}_2 \rightarrow \mathcal{W}_n$ for the sequence of the covering relations:
\[
    \mathcal{N}_2 \xRightarrow{P} \mathcal{N}_{19} \xRightarrow{P} \mathcal{N}_{20} \xRightarrow{P} ... \xRightarrow{P} \mathcal{N}_{29} \xRightarrow{P}
    \mathcal{M}_0 \xRightarrow{P} \mathcal{M}_1 \xRightarrow{P} ... \xRightarrow{P} \mathcal{M}_{n-1} \xRightarrow{P} \mathcal{W}_n,
\]
which is the combination of the sequences of the covering relations (\ref{eq:covrel_periodic_to_infinity}), (\ref{eq:covrel_jump_to_infinity}) and the sequence of the covering relations described in the Corollary \ref{col:covrel_M_W}. We introduce the notation $\mathcal{N}_2 \rightarrow \mathcal{Q}_{4k}$ for the sequence of the covering relations:
\[
    \mathcal{N}_2 \xRightarrow{P} ... \xRightarrow{P} \mathcal{N}_K \xRightarrow{P} ... \xRightarrow{P} S\mathcal{N}_4 \xRightarrow{P} \mathcal{R}_1 \xRightarrow{P} \mathcal{R}_2 \xRightarrow{P} ... \xRightarrow{P} \mathcal{Q}_{4k},
\]
which is the combination of the suitable contiguous subsequences of the covering relations (\ref{eq:covering_sequence_o1}), (\ref{eq:covering_sequence_o2}) and the sequences of the covering relations (\ref{eq:R-Q-coverings}) and (\ref{eq:gluing}).
Time reversal symmetry of the system implies the existence of the sequence of the covering relations
\[
    \mathcal{W}_n \rightarrow \mathcal{N}_2 \quad\text{and}\quad \mathcal{Q}_{4k} \rightarrow \mathcal{N}_2.
\]

For each set of orbits $\mathcal{O}^+_H, \mathcal{O}^+_P, \mathcal{O}^+_{OI}, \mathcal{O}^+_{OC}, \mathcal{O}^+_{OB}, \mathcal{O}^+_B, \mathcal{O}^+_C$ there exists a sequence of covering relations which begins at the h-set $\mathcal{N}_2$, and which implies the existence of the trajectory which belongs to that set of orbits:
\begin{itemize}
    \item for $\mathcal{O}_H^+$:
    \[
        \mathcal{N}_2 \xRightarrow{P} \mathcal{N}_{19} \xRightarrow{P} \mathcal{N}_{20} \xRightarrow{P} ... \xRightarrow{P} \mathcal{N}_{29} \xRightarrow{P}
        \mathcal{M}_0,
    \]
    which is the consequence of Theorem \ref{thm:covrel_periodic_to_infinity}, Theorem \ref{thm:covrel_jump_to_infinity} and Theorem \ref{lem:hyperbolic_evolution};
    \item for $\mathcal{O}_P^+$:
    \[
        \mathcal{N}_2 \xRightarrow{P} \mathcal{N}_{19} \xRightarrow{P} \mathcal{N}_{20} \xRightarrow{P} ... \xRightarrow{P} \mathcal{N}_{29} \xRightarrow{P}
        \mathcal{M}_0 \xRightarrow{P} \mathcal{M}_1 \xRightarrow{P} ...
    \]
    which is the consequence of Theorem \ref{thm:covrel_periodic_to_infinity}, Theorem \ref{thm:covrel_jump_to_infinity} and Corollary \ref{col:covrel_M_W};
    \item for $\mathcal{O}_{OI}^+$:
    \[
        \mathcal{N}_2 \rightarrow \mathcal{W}_1 \rightarrow \mathcal{N}_2 \rightarrow \mathcal{W}_2 \rightarrow \mathcal{N}_2 \rightarrow \mathcal{W}_3 \rightarrow ...
    \]
    \item for $\mathcal{O}_{OC}^+$:
    \[
        \mathcal{N}_2 \rightarrow \mathcal{Q}_4 \rightarrow \mathcal{N}_2 \rightarrow \mathcal{Q}_8 \rightarrow \mathcal{N}_2 \rightarrow \mathcal{Q}_{12} \rightarrow ...
    \]
    \item for $\mathcal{O}_{OB}^+$:
    \[
        \mathcal{N}_2 \rightarrow \mathcal{W}_1 \rightarrow \mathcal{N}_2 \rightarrow \mathcal{Q}_4 \rightarrow \mathcal{N}_2 \rightarrow \mathcal{W}_2 \rightarrow \mathcal{N}_2 \rightarrow \mathcal{Q}_8 \rightarrow ...
    \]
    \item for $\mathcal{O}_B^+$:
    \[
        \mathcal{N}_2 \xRightarrow{P} ... \xRightarrow{P} \mathcal{N}_K \xRightarrow{P} ... \xRightarrow{P} \mathcal{N}_2 \xRightarrow{P} ... ,
    \]
    which the combination of the suitable contiguous subsequences of (\ref{eq:covering_sequence_o1}) and (\ref{eq:covering_sequence_o2})
    \item for $\mathcal{O}_C^+$:
    \[
        \mathcal{N}_2 \xRightarrow{P} \mathcal{N}_3 \xRightarrow{P} \mathcal{N}_0,
    \]
    what is the consequence of Theorem \ref{thm:cover-discs} and the fact that the collision manifold is both a horizontal and a vertical disc of the h-sets $\mathcal{N}_0$ (see Definitions \ref{def:horizontal-disc} and \ref{def:vertical-disc}).
\end{itemize}

The time reversal symmetry of the system implies that for each set of orbits $\mathcal{O}^-_H$, $\mathcal{O}^-_P$, $\mathcal{O}^-_{OI}$, $\mathcal{O}^-_{OC}$, $\mathcal{O}^-_{OB}$, $\mathcal{O}^-_B$, $\mathcal{O}^-_C$ there exists a sequence of covering relations which ends at $\mathcal{N}_2$, and which implies the existence of the trajectory which belongs to that set of orbits.

An arbitrary sequence from the first group can be connected to an arbitrary set from the second group which implies the existence of the orbit which belongs to $\mathcal{O}_X^- \cap \mathcal{O}_Y^+$ for any $X,Y \in \{H,P,OI,OC,OB,B,C\}$. 

The sequence of the covering relations
\[
    \mathcal{W}_n \rightarrow \mathcal{N}_2 \rightarrow \mathcal{Q}_{4k} \rightarrow \mathcal{N}_2 \rightarrow \mathcal{W}_n
\]
implies the existence of the periodic orbit, which for the sufficiently large $n$ and $k$ attains an arbitrarily small but positive distance from the collision and an arbirarily large distance from the collision.

The sequence of the covering relations
\[
    \mathcal{N}_0 \xRightarrow{P}
    \mathcal{N}_1 \xRightarrow{P}
    \mathcal{N}_2 \rightarrow \mathcal{W}_n \rightarrow
    \mathcal{N}_2 \xRightarrow{P}
    \mathcal{N}_3 \xRightarrow{P}
    \mathcal{N}_0
\]
implies the existence of the ejection/collision orbit which for the sufficiently large $n$ is arbitrarily large (i.e. reaches arbitrarily far from the center of mass of the system), which concludes our proof.

\section{Conclusions}
We have shown an explicit construction, which leads the connections between the forward and backward time motions exhibiting: a collision, spatially bounded oscillatory motion to collision, spatially bounded motion away from collision, oscillatory motion to infinity bounded away from the collision, oscillatory motion to both collision and infinity, parabolic motion to infinity and the hyperbolic motion to infinity in the Earth-Moon planar circular restricted three body problem. The construction was based on exploiting the homoclinic connections for the family of Lyapunov orbits, which persist in the Levi-Civita coordinates for the ejection/collision orbit and the heteroclinic connection from the ejection/collision orbit to the limit set of the parabolic motion, which persists in the McGehee coordinates. 

We use the Earth-Moon system as an example, but we believe that the method is applicable to other mass parameters, provided that the Lyapunov family collides with one of the primaries. We emphasise though that as the mass parameter is changed, the computer assisted proof needs to be re-run and most likely adjusted. In particular, the method of covering relations relies on the good alignment of the local coordinates, which changes with the change of parameter.

This work is the realization of the target which was set in \cite{CAPINSKI2026109173}, where we discussed about the possibility of the combination of the results contained in that paper with the results from \cite{MR4391693}, which involved the hyperbolic, parabolic and oscillatory motions to infinity. As a consequence, we extend the results from \cite{MR4890942}, which are applicable to sufficiently small $\mu$. Such outcome goes towards the understanding the evolution of the three-body problem over long timescales, which is an important question in dynamical systems theory.

\appendix

\section{Coordinate systems}
\label{appendix:coordinate_systems}

Let $\epsilon:= 8.5 \cdot 10^{-10}$. This coefficient will play the role of a rescaling factor for our coordinate changes. Its role is to have the h-sets $\mathcal{N}_0,\ldots, \mathcal{N}_{29}$ in the local coordinates given by $\psi_0,\ldots, \psi_{29}$, respectively, to be of the same size and of the form $N_c=[-1,1]^2\subset \mathbb{R}^2$. In short, the h-sets in the local coordinates will be of order one, but their size in the coordinates of the PCR3BP will be of order $10^{-10}$.
 
Let $\mathcal{U}$ denote the vector normalization i.e. $\mathcal{U}(w) = w / \| w \|$.  We choose the matrices $A_k\in\mathbb{R}^{4\times4}$ for the definition of the coordinate changes $\psi_k$ (see (\ref{eq:local-parametrization-choice}) and (\ref{eq:A-form})) to be
\begin{equation}
    A_k:= \, \left[
        \epsilon\,\hat{u}_k \quad
        \epsilon\,\hat{s}_k \quad
        \epsilon\,\mathcal{U} \left( J\nabla \Gamma(w_k) \right) \quad
        \epsilon\,\mathcal{U} \left( \nabla \Gamma(w_k) \right) \right], \quad \mbox{for }k=1,\ldots,29, \label{eq:Ak-choice}
\end{equation}
where the vectors $w_k$, $\hat{u}_k$ and $\hat{s}_k$ are chosen as written out in Tables \ref{table:wk}, \ref{table:uk} and \ref{table:sk}, respectively. Note that the choice of $A_k$ in (\ref{eq:Ak-choice}) is of the form (\ref{eq:A-form}); with a rescaling of the columns. The matrices $A_{19}, ..., A_{29}$ are chosen according to (\ref{eq:A_19_29}).

We use the matrices $A_k$ from (\ref{eq:Ak-choice}) to define $\psi_k$ using (\ref{eq:local-parametrization-choice}) for $k=1,\ldots,29$.

\begin{table}[H]
    \centering
    \scriptsize
    \begin{tabular}{lllll}
        \hline
         & $u$ & $\phantom{-}v$ & $\phantom{-}p_u$ & $\phantom{-}p_v$ \\
        \hline
        $w_{0}$ & $\phantom{-}0$ & $\phantom{-}0$ & $\phantom{-}0$ & $\phantom{-}2.81112771399$ \\
        $w_{1}$ & $\phantom{-}1.01939532911$ & $\phantom{-}0.0352455355111$ & $-1.57992203844$ & $-0.220188078268$ \\
        $w_{2}$ & $\phantom{-}0.997511358555$ & $\phantom{-}0$ & $\phantom{-}0$ & $-3.10035066234$ \\
        $w_{3}$ & $\phantom{-}1.01939532911$ & $-0.0352455355111$ & $\phantom{-}1.57992203844$ & $-0.220188078268$ \\
        $w_{4}$ & $\phantom{-}1.82311329298$e-10 & $\phantom{-}0$ & $\phantom{-}3.0389849377$e-09 & $\phantom{-}2.81112771399$ \\
        $w_{5}$ & $\phantom{-}1.01939364734$ & $\phantom{-}0.0352429461228$ & $-1.57993713742$ & $-0.220199770725$ \\
        $w_{6}$ & $\phantom{-}0.997511359537$ & $\phantom{-}0$ & $\phantom{-}3.81728719867$e-07 & $-3.1003514216$ \\
        $w_{7}$ & $\phantom{-}1.01939365573$ & $-0.0352429413345$ & $\phantom{-}1.57993769304$ & $-0.22019938229$ \\
        $w_{8}$ & $\phantom{-}1.23284045903$e-07 & $\phantom{-}0$ & $\phantom{-}2.05626985632$e-06 & $\phantom{-}2.81112771099$ \\
        $w_{9}$ & $\phantom{-}1.01939456403$ & $\phantom{-}0.0352429632775$ & $-1.57993456608$ & $-0.220196017823$ \\
        $w_{10}$ & $\phantom{-}0.997512022626$ & $\phantom{-}0$ & $\phantom{-}0.000258254271245$ & $-3.10086432733$ \\
        $w_{11}$ & $\phantom{-}1.01940023627$ & $-0.0352397234978$ & $\phantom{-}1.58031042535$ & $-0.219933213027$ \\
        $w_{12}$ & $\phantom{-}8.34102256631$e-05 & $\phantom{-}0$ & $\phantom{-}0.00139148280123$ & $\phantom{-}2.8111253283$ \\
        $w_{13}$ & $\phantom{-}1.02000949836$ & $\phantom{-}0.0352470440096$ & $-1.57821685002$ & $-0.21769836291$ \\
        $w_{14}$ & $\phantom{-}0.997948108217$ & $\phantom{-}0$ & $\phantom{-}0.196225509832$ & $-3.49133960833$ \\
        $w_{15}$ & $\phantom{-}1.023984317$ & $-0.0326239190489$ & $\phantom{-}1.83944484044$ & $-0.0123000584495$ \\
        $w_{16}$ & $\phantom{-}0.0828113236996$ & $-1.13042936454$ & $\phantom{-}0.0931363222648$ & $\phantom{-}1.32036952924$ \\
        $w_{17}$ & $\phantom{-}0.0742904549562$ & $\phantom{-}0$ & $\phantom{-}1.24626303185$ & $\phantom{-}2.51265379059$ \\
        $w_{18}$ & $\phantom{-}1.26583072872$ & $\phantom{-}0$ & $\phantom{-}0$ & $\phantom{-}0.120135068527$ \\
        $w_{19}$ & $\phantom{-}1.01946005922$ & $-0.0353451461152$ & $\phantom{-}1.57934194996$ & $-0.219739806648$ \\
        $w_{20}$ & $\phantom{-}8.05409106162$e-10 & $\phantom{-}0.000120876807692$ & $-2.93479781138$e-06 & $\phantom{-}2.81112769945$ \\
        $w_{21}$ & $\phantom{-}1.01933056828$ & $\phantom{-}0.0351459199114$ & $-1.5805038784$ & $-0.220639737798$ \\
        $w_{22}$ & $\phantom{-}0.997514967299$ & $-0.00021753533232$ & $\phantom{-}0.16783494776$ & $-3.09301169811$ \\
        $w_{23}$ & $\phantom{-}1.01946010464$ & $-0.0353451286972$ & $\phantom{-}1.57934436095$ & $-0.219738054317$ \\
        $w_{24}$ & $\phantom{-}5.44613493013$e-07 & $\phantom{-}0.000118282132727$ & $\phantom{-}6.19225850923$e-06 & $\phantom{-}2.81112768679$ \\
        $w_{25}$ & $\phantom{-}1.01933301393$ & $\phantom{-}0.0351435326991$ & $-1.58050695231$ & $-0.220634296622$ \\
        $w_{26}$ & $\phantom{-}0.997518133399$ & $-0.000223007181604$ & $\phantom{-}0.173506906544$ & $-3.09488025629$ \\
        $w_{27}$ & $\phantom{-}1.01949083414$ & $-0.0353333277483$ & $\phantom{-}1.5809755598$ & $-0.218551655939$ \\
        $w_{28}$ & $\phantom{-}0.000364884317097$ & $-0.00163993493724$ & $\phantom{-}0.00619232212021$ & $\phantom{-}2.81110927758$ \\
        $w_{29}$ & $\phantom{-}1.02098553866$ & $\phantom{-}0.0335275406906$ & $-1.58257591259$ & $-0.21661007693$ \\
    \end{tabular}
    \caption{Approximate values of the components of the vectors $w_k$.\label{table:wk}}
\end{table}

\begin{table}[H]
    \scriptsize
    \centering
    \begin{tabular}{lllll}
        \hline
         & $\phantom{-}u$ & $\phantom{-}v$ & $\phantom{-}p_u$ & $\phantom{-}p_v$ \\
        \hline
        $\hat{u}_{0}$ & $\phantom{-}0.0598653990819$ & $\phantom{-}0.0243392345095$ & $\phantom{-}0.99790861489$ & $-0.00146013168443$ \\
        $\hat{u}_{1}$ & $\phantom{-}0.0512722628998$ & $-0.0539645641977$ & $-0.0790749932687$ & $\phantom{-}0.106664891545$ \\
        $\hat{u}_{2}$ & $\phantom{-}0.0123990872795$ & $\phantom{-}0.00623386910969$ & $\phantom{-}8.06085056107$e-06 & $-9.58885003297$ \\
        $\hat{u}_{3}$ & $\phantom{-}0.0655510635754$ & $-0.0519328405322$ & $\phantom{-}0.994851864788$ & $\phantom{-}1.26355518097$ \\
        $\hat{u}_{4}$ & $\phantom{-}0.0598653990819$ & $\phantom{-}0.0243392345095$ & $\phantom{-}0.99790861489$ & $-0.00146013168443$ \\
        $\hat{u}_{5}$ & $\phantom{-}0.0625757321266$ & $-0.0658560755461$ & $-0.0965279955578$ & $\phantom{-}0.13018905303$ \\
        $\hat{u}_{6}$ & $\phantom{-}0.0184674816754$ & $\phantom{-}0.00928486590201$ & $\phantom{-}1.20031980776$e-05 & $-14.2818597745$ \\
        $\hat{u}_{7}$ & $\phantom{-}0.119128113811$ & $-0.0943647041922$ & $\phantom{-}1.80852788038$ & $\phantom{-}2.2967467384$ \\
        $\hat{u}_{8}$ & $\phantom{-}0.132804359411$ & $\phantom{-}0.0539941945251$ & $\phantom{-}2.21374370854$ & $-0.0032407835347$ \\
        $\hat{u}_{9}$ & $\phantom{-}0.138815104952$ & $-0.146096387381$ & $-0.214131083267$ & $\phantom{-}0.288817842031$ \\
        $\hat{u}_{10}$ & $\phantom{-}0.0409636350452$ & $\phantom{-}0.0205951178031$ & $\phantom{-}2.24064629998$e-05 & $-31.6930288739$ \\
        $\hat{u}_{11}$ & $\phantom{-}0.264355726553$ & $-0.209336295197$ & $\phantom{-}4.01213882195$ & $\phantom{-}5.0966647523$ \\
        $\hat{u}_{12}$ & $\phantom{-}0.294792860317$ & $\phantom{-}0.120543706008$ & $\phantom{-}4.91486295091$ & $-0.00967322389548$ \\
        $\hat{u}_{13}$ & $\phantom{-}0.305475938325$ & $-0.327843230549$ & $-0.468071313644$ & $\phantom{-}0.654208911006$ \\
        $\hat{u}_{14}$ & $\phantom{-}0.0842120960402$ & $\phantom{-}0.0422016823499$ & $-0.00607062017782$ & $-89.0241271111$ \\
        $\hat{u}_{15}$ & $\phantom{-}0.731763394696$ & $-0.458286780463$ & $\phantom{-}8.89272504938$ & $\phantom{-}13.9801270312$ \\
        $\hat{u}_{16}$ & $-3.39472774076$ & $-2.60519046749$ & $\phantom{-}6.3830932378$ & $-4.93809685454$ \\
        $\hat{u}_{17}$ & $\phantom{-}0.504441646028$ & $\phantom{-}0.387143559261$ & $\phantom{-}4.32240508069$ & $-2.22878232827$ \\
        $\hat{u}_{18}$ & $\phantom{-}0.0497937735344$ & $-0.534053889757$ & $\phantom{-}0.326462982573$ & $\phantom{-}0.0814565811784$ \\
        \hline
    \end{tabular}
    \caption{Approximate values of the components of the vectors $\hat{u}_k$.\label{table:uk}}
\end{table}

\begin{table}[H]
	\scriptsize
    \centering
    \begin{tabular}{lllll}
        \hline
         & $\phantom{-}u$ & $\phantom{-}v$ & $\phantom{-}p_u$ & $\phantom{-}p_v$ \\
        \hline
        $\hat{s}_{0}$ & $\phantom{-}0.0598653990819$ & $-0.0243392345095$ & $-0.99790861489$ & $-0.00146013168443$ \\
        $\hat{s}_{1}$ & $\phantom{-}0.0655510635754$ & $\phantom{-}0.0519328405322$ & $-0.994851864788$ & $\phantom{-}1.26355518097$ \\
        $\hat{s}_{2}$ & $\phantom{-}0.0123990872795$ & $-0.00623386910969$ & $-8.06085056107$e-06 & $-9.58885003297$ \\
        $\hat{s}_{3}$ & $\phantom{-}0.0512722628998$ & $\phantom{-}0.0539645641977$ & $\phantom{-}0.0790749932687$ & $\phantom{-}0.106664891545$ \\
        $\hat{s}_{4}$ & $\phantom{-}0.0598653990819$ & $-0.0243392345095$ & $-0.99790861489$ & $-0.00146013168443$ \\
        $\hat{s}_{5}$ & $\phantom{-}0.0536845559625$ & $\phantom{-}0.042525059308$ & $-0.815005301074$ & $\phantom{-}1.03501866893$ \\
        $\hat{s}_{6}$ & $\phantom{-}0.00832229054357$ & $-0.00418418326664$ & $-5.41171804734$e-06 & $-6.4360581837$ \\
        $\hat{s}_{7}$ & $\phantom{-}0.0281994747078$ & $\phantom{-}0.0296777546137$ & $\phantom{-}0.0434999318962$ & $\phantom{-}0.058669162844$ \\
        $\hat{s}_{8}$ & $\phantom{-}0.026978030238$ & $-0.0109685022423$ & $-0.449703262016$ & $-0.000657678209258$ \\
        $\hat{s}_{9}$ & $\phantom{-}0.0242011480547$ & $\phantom{-}0.0191715734179$ & $-0.367373745058$ & $\phantom{-}0.466569759349$ \\
        $\hat{s}_{10}$ & $\phantom{-}0.0037506919161$ & $-0.00188532541287$ & $-2.82255343343$e-06 & $-2.9018616866$ \\
        $\hat{s}_{11}$ & $\phantom{-}0.0127077970338$ & $\phantom{-}0.0133767067373$ & $\phantom{-}0.0196090894894$ & $\phantom{-}0.0264540615654$ \\
        $\hat{s}_{12}$ & $\phantom{-}0.0121439073914$ & $-0.0049838504505$ & $-0.202713054229$ & $-0.000197805387775$ \\
        $\hat{s}_{13}$ & $\phantom{-}0.0112900526678$ & $\phantom{-}0.00929003414053$ & $-0.161456580967$ & $\phantom{-}0.211692282532$ \\
        $\hat{s}_{14}$ & $\phantom{-}0.00144550242937$ & $-0.000618072920234$ & $-0.000105478869183$ & $-1.52819846002$ \\
        $\hat{s}_{15}$ & $\phantom{-}0.00455638237545$ & $\phantom{-}0.00563436828252$ & $\phantom{-}0.00813263074882$ & $\phantom{-}0.0134836469778$ \\
        $\hat{s}_{16}$ & $\phantom{-}0.0299785746856$ & $\phantom{-}0.00891662895616$ & $-0.0343147668569$ & $\phantom{-}0.0340399244557$ \\
        $\hat{s}_{17}$ & $-0.026049583212$ & $-0.0475477618468$ & $-0.409735409863$ & $\phantom{-}0.20805798649$ \\
        $\hat{s}_{18}$ & $\phantom{-}0.0497937735344$ & $\phantom{-}0.534053889757$ & $-0.326462982573$ & $\phantom{-}0.0814565811784$ \\
        \hline
    \end{tabular}
    \caption{Approximate values of the components of the vectors $\hat{s}_k$.}
    \label{table:sk}
\end{table}

\bibliographystyle{ieeetr} 
\bibliography{references}

\end{document}